\documentclass[11pt]{amsart}

	\usepackage{amscd,amsfonts,amsmath,amssymb,amstext,amsthm,amsxtra,subcaption}
	\usepackage{graphicx}
	\usepackage{tikz}

	\usetikzlibrary{shapes.geometric,arrows}
	\usepackage{romannum,xcolor,comment}
	\usepackage[CJKbookmarks=true,bookmarks=true]{hyperref}
	\hypersetup
	{
		pdftitle={On embeddings of the Grassmannian {\texorpdfstring{\protect $Gr(2,m)$}{Gr(2,m)}} into the Grassmannian {\texorpdfstring{\protect $Gr(2,n)$}{Gr(2,n)}}},
		pdfauthor={Minhyuk Kwon},
		pdfkeywords={Complex Grassmannians, Holomorphic embeddings, Schubert cycles, Chern classes},
	}
		\numberwithin{equation}{section}

		\theoremstyle{plain}
		\newtheorem{theorem}{Theorem}[section]
		\newtheorem{lemma}[theorem]{Lemma}
		\newtheorem{proposition}[theorem]{Proposition}
		\newtheorem{corollary}[theorem]{Corollary}

		\numberwithin{theorem}{section}

		\theoremstyle{definition}
		\newtheorem{definition}[theorem]{Definition}
		\newtheorem*{question}{Question}
		\newtheorem{example}[theorem]{Example}
		\newtheorem{convention}[theorem]{Convention}
		\newtheorem{notation}[theorem]{Notation}

		\theoremstyle{remark}
		\newtheorem{remark}[theorem]{Remark}
		\newtheorem{note}[theorem]{Note}

		\newcommand{\thistheoremname}{}

		\theoremstyle{plain}
		\newtheorem{genericthm}[theorem]{\thistheoremname}

		\theoremstyle{plain}
		\newtheorem*{genericthm*}{\thistheoremname}
		\newenvironment{namedthm*}[1]
		{
			\renewcommand{\thistheoremname}{#1}
			\par
			\begin{genericthm*} 
			\par
			\ignorespaces
		}
		{
			\end{genericthm*}
			\par
			\ignorespacesafterend
		}

		\newenvironment{eqnma}
		{
			\begin{equation}
				\begin{aligned}
				\ignorespaces
		}
		{
				\end{aligned}
			\end{equation}
			\ignorespacesafterend
		}

		\newenvironment{eqnmg}
		{
			\begin{equation}
				\begin{gathered}
				\ignorespaces
		}
		{
				\end{gathered}
			\end{equation}
			\ignorespacesafterend
		}

		\newcommand{\Aut}{\operatorname{Aut}}

		\newcommand{\rank}{\operatorname{rank}}
		\newcommand{\sspan}{\operatorname{span}}

		\newcommand{\ckE}{\check{E}}
		\newcommand{\ckX}{\check{X}}
		\newcommand{\Zbb}{\mathbb{Z}}
		\newcommand{\Rbb}{\mathbb{R}}
		\newcommand{\Cbb}{\mathbb{C}}
		\newcommand{\Fbb}{\mathbb{F}}
		\newcommand{\Pbb}{\mathbb{P}}
		\newcommand{\PGL}{\mathbf{P}GL}
		\newcommand{\quotient}[2]{\left.\raisebox{.1em}{$#1$}\middle/\raisebox{-.1em}{$#2$}\right.}
		\newcommand{\centermark}[2]{\hskip 0.1cm \raisebox{-0.#1cm}{#2}}
		\newcommand{\bds}{\boldsymbol}
		\newcommand{\Implies}{\,\Rightarrow\,}

		\tikzstyle{startstop} = [rectangle, rounded corners, minimum width=3cm, minimum height=1cm, text centered, draw=black, fill=red!30]
		\tikzstyle{io} = [trapezium, trapezium left angle=70, trapezium right angle=110, minimum width=3cm, minimum height=1cm, text centered, draw=black, fill=blue!30]
		\tikzstyle{process} = [rectangle, minimum width=3cm, minimum height=1cm, text centered, text width=3cm, draw=black, fill=orange!30]
		\tikzstyle{decision} = [diamond, minimum width=3cm, minimum height=1cm, text centered, draw=black, fill=green!30]
		\tikzstyle{arrow} = [thick,->,>=stealth]

	\newcommand{\omgao}{\omega_{1, 0}}
	\newcommand{\omgbo}{\omega_{2, 0}}
	\newcommand{\omgaa}{\omega_{1, 1}}
	\newcommand{\omgko}[1]{\omega_{#1, 0}}
	\newcommand{\tlomgao}{\tilde{\omega}_{1, 0}}
	\newcommand{\tlomgbo}{\tilde{\omega}_{2, 0}}
	\newcommand{\tlomgaa}{\tilde{\omega}_{1, 1}}
	\newcommand{\cohomgr}[2]{H^{#2} (Gr(2,#1), \Zbb)}

	\newcommand{\quotientringao}{\quotient{\Zbb [\omgao]}{(\omgao^{m-1})}}
	\newcommand{\quotientringaa}{\quotient{\Zbb [\omgaa]}{(\omgaa^{\lfloor (m-2)/2 \rfloor + 1})}}

	\newcommand{\lmod}{%%
		\preceq
	}

	\title[embeddings of $Gr(2,m)$ into $Gr(2,n)$]{On embeddings of the Grassmannian $Gr(2,m)$ into the Grassmannian $Gr(2,n)$}

	\author[M. Kwon]{Minhyuk Kwon}
	\address{Department of Mathematical Sciences \\
		Seoul National University\\
		1 Gwanak-ro Gwanak-gu Seoul 08826, Korea}
	\email{mu10836@snu.ac.kr}


\begin{document}
	
	\pagenumbering{arabic}
	\setcounter{page}{1}
	
	\begin{abstract}
		In this {paper}, we consider holomorphic embeddings of $Gr(2,m)$ into $Gr(2,n)$. We study such embeddings by finding all possible total Chern classes of the pull-back of the universal bundles under these embeddings. Using the relations between Chern classes of the universal bundles and Schubert cycles, together with properties of holomorphic vector bundles of rank $2$ on complex Grassmannians, we find conditions on $m$ and $n$ for which any holomorphic embedding of $Gr(2,m)$ into $Gr(2,n)$ is linear.
	\end{abstract}

	\maketitle
	
	%\footnote{\noindent M. Kwon: Department of Mathematical Sciences, Seoul National University, 1 Gwanak-ro Gwanak-gu Seoul 08826, Korea. e-mail:mu10836@snu.ac.kr\\
	%	Mathematics Subject Classifications (2010): 14M15, 51E25\\
	%	Keywords: Grassmannian, Schubert varieties, embedding}

	\section{Introduction} \label{sec1}
		
		Let $d, \, m$ be positive integers with $d < m$.
		For an $m$-dimensional complex vector space $V$, a complex Grassmannian $Gr(d,V)$ is the space parameterizing all $d$-dimensional subspaces of $V$.
		Since the description of $Gr(d,V)$ is concrete and explicit, and concerns matrices and vector spaces over the complex field $\Cbb$, many practical and computable techniques to study it have been developed.
		%A real Grassmannian $Gr_{\Rbb} (d,V)$ and complex Grassmannians $Gr_{\Cbb} (d,V)$ are typical.
		%However, we can choose another field as a base field of arbitrary characteristic, and the structure of a Grassmannian $Gr (d,V)$ depends substantially on $k$.
		There are algebraic varieties which are generalized from $Gr(d,V)$, for instance,
		an orthogonal Grassmannian $Gr_{q} (d,V)$,
		a symplectic Grassmannian $Gr_{\omega} (d,V)$ and
		a flag variety $F(d_1, d_2, \cdots , d_{k-1}, m)$.
		In addition, after replacing $\Cbb$ by another field $\Fbb$, it is possible to construct another Grassmannian $Gr_{\Fbb} (d,V)$ and its structure depends substantially on the base field $\Fbb$.
		For this reason, many mathematicians have been interested in complex Grassmannians with their generalizations, and have studied them from various perspectives and purposes.

		A complex Grassmannian is a non-singular projective variety, and a fundamental and significant object of algebraic geometry.
		Many features of complex Grassmannians, including homology and the cohomology groups, automorphism groups, holomorphic embeddings of them into complex projective spaces and defining ideals, are well-known.
		In particular, a complex Grassmannian $Gr(d,V)$ admits a cell decomposition.
		The closure of each cell is called a \emph{Schubert variety} and Schubert varieties play an important role in understanding $Gr(d,V)$.
		Using Poincar\'{e} duality, the homology class of each Schubert variety corresponds to the cohomology class, which is called a \emph{Schubert cycle}, and Schubert cycles on $Gr(d,V)$ can be classified by the $d$-tuples of non-negative integers satisfying some inequality.
		The set of all Schubert cycles forms a $\Zbb$-module basis of the cohomology ring of $Gr(d,V)$, and the multiplications of Schubert cycles are determined by a combinatorial rule, namely Pieri's formula.

		There are two canonical holomorphic vector bundles on $Gr(d,V)$, the \emph{universal bundle $E(d,V)$} and the \emph{universal quotient bundle $Q(d,V)$}, which are defined in natural ways: the fiber of $E(d,V)$ at $x \in Gr(d,V)$ is given by
		\begin{center}
			the $d$-dimensional subspace $L_x$ of $V$ which corresponds to $x$,
		\end{center}
		and the fiber of $Q(d,V)$ at $x$ is given by the quotient space $\quotient{V}{L_x}$.
		Every Chern class of $E(d,V)$ and $Q(d,V)$ is a Schubert cycle (up to signs) and the cohomology ring of $Gr(d,V)$ is generated by the set of all Chern classes of $E(d,V)$ as a ring.
		Furthermore, the tangent bundle of $Gr(d,V)$ is isomorphic to $\ckE (d,V) \otimes Q(d,V)$ where $\ckE (d,V)$ is the dual bundle of $E(d,V)$.
		When $V = \Cbb^m$, we denote $Gr(d,V)$, $E(d,V)$ and $Q(d,V)$ simply by $Gr(d,m)$, $E(d,m)$ and $Q(d,m)$, respectively.
		%Given a holomorphic map $\varphi \colon Gr(d_1,m) \hookrightarrow Gr(d_2,n)$, the fiber of $\varphi^* (E(d_2,n))$ at $x \in Gr(d_1,m)$ equals the fiber of $E(d_2,n)$ at $\varphi (x)$ which is the $d_2$-dimensional subspace $L_{\varphi (x)}$ of $\Cbb^n$ corresponding to $\varphi (x)$.
		%Thus any holomorphic map $\varphi \colon Gr(d_,m) \hookrightarrow Gr(d_2)$ is determined completely by the pullback bundle $\varphi^* (E(d_2,n))$ under the map $\varphi$.
		%More elaborate relation between holomorphic maps into a complex Grassmannian and holomorphic vector bundles of fixed rank will be provided later.

		In this {paper}, we discuss holomorphic embeddings between complex Grassmannians.
		For any $d_1 < m$ and $d_2 < n$ with $d_1 \le d_2$ and $m - d_1 \le n - d_2$, there is a natural holomorphic embedding of $Gr(d_1,m)$ into $Gr(d_2,n)$:
		\begin{enumerate}
			\item[\fbox{1}] Let $f \colon \Cbb^m \hookrightarrow \Cbb^n$ be an injective linear map and let $W$ be a $(d_2-d_1)$-dimensional subspace of $\Cbb^n$ satisfying $W \cap f(\Cbb^m) = 0$;
			\item[\fbox{2}] the pair $(f,W)$ induces a holomorphic embedding $\tilde{f}_W \colon Gr(d_1,m) \hookrightarrow Gr(d_2,n)$ which is given by
			\[
				L_{\tilde{f}_W(x)} := f(L_x) \oplus W, \qquad x \in Gr(d_1,m).
			\]
		\end{enumerate}
		We call such an embedding $\tilde{f}_W$ to be \emph{linear}.
		Consider the following question:

		\begin{question} \hypertarget{question on the linearity}{}
			For $d_1 < m$ and $d_2 < n$ with $d_1 \le d_2$ and $m - d_1 \le n - d_2$, what is a sufficient condition for the linearity of holomorphic embeddings of $Gr(d_1,m)$ into $Gr(d_2,n)$?
			More generally, how can we classify such embeddings?
		\end{question}

		The most fundamental answer to \hyperlink{question on the linearity}{Question} is about the case when $d_1 = d_2 \, (=: d)$ and $m = n$.
		In this case, a holomorphic embedding $\varphi \colon Gr(d,m) \hookrightarrow Gr(d,m)$ is an automorphism of $Gr(d,m)$.
		To describe a non-linear automorphism of $Gr(d,m)$, fix a basis $\mathcal{B} := \{ e_1, \cdots , e_m \}$ of $\Cbb^m$ and let $\{ e_1^*, \cdots e_m^* \}$ be the dual basis of $\mathcal{B}$.
		The choice of $\mathcal{B}$ induces a linear isomorphism $\iota \colon (\Cbb^m)^* \to \Cbb^m$ which is determined by  $\iota (e_j^*) = e_j$ for all $1 \le j \le m$.
		Define a map $\phi \colon Gr(d,m) \to Gr(m-d,m)$ by
		\begin{equation} \label{definition of dual map}
			L_{\phi (x)} := \iota \left( L_x^{\perp} \right), \qquad x \in Gr(d,m)
		\end{equation}
		where $L_x^{\perp} \subset (\Cbb^m)^*$ is the annihilator of $L_x$.
		We call such a map $\phi$ a \emph{dual map}.
		In particular, when $m=2d$, a dual map $\phi \colon Gr(d,2d) \to Gr(d,2d)$ is an automorphism.
		In \cite{CH}, W.-L. Chow classified all automorphisms of $Gr(d,m)$ and showed that every automorphism of $Gr(d,m)$ is linear except when $m=2d$.

		\begin{theorem} [{\cite[Theorem \Romannum{11} and \Romannum{15}]{CH}}] \label{automorphism groups of Grassmannians}
			The automorphism group of $Gr(d,m)$ is
			\[
				\Aut (Gr(d,m)) =
				\begin{cases}
					\PGL (m, \Cbb), & \text{if } m \neq 2d\\
					\begin{tabular}{l}
						$\PGL (2d, \Cbb) \sqcup (\phi \circ \PGL (2d, \Cbb))$\\
						$= \PGL (2d, \Cbb) \sqcup (\PGL (2d, \Cbb) \circ \phi)$
					\end{tabular}
					, & \text{if } m=2d
				\end{cases}
			\]
			where $\phi \colon Gr(d,2d) \to Gr(d,2d)$ is a dual map.
		\end{theorem}

		%When either $m < n$ or $d_1 < d_2$, there have been several approaches to answer \hyperlink{question on the linearity}{Question}.

		In \cite{MO}, N. Mok considered holomorphic embeddings $\varphi \colon Gr(d_1,m) \hookrightarrow Gr(d_2,n)$ with $2 \le d_1 \le d_2$ and $2 \le m-d_1 \le n-d_2$, and obtained a geometric condition on $\varphi$ for the linearity.
		%(N. Mok used the term `\emph{standard}' instead of our term `\emph{linear}').
		For $x \in Gr(d_1,m)$, we regard each tangent vector of $Gr(d_1,m)$ at $x$ as an element of $\ckE (d_1,m)_x \otimes Q(d_1,m)_x$ where $\mathcal{E}_x$ denotes the fiber of $\mathcal{E}$ at $x$.
		We call a tangent vector of $Gr(d_1,m)$ at $x$ to be \emph{decomposable} if it can be written as $v \otimes w$ for some $v \in \ckE (d_1,m)_x$ and $w \in Q(d_1,m)_x$.
		N. Mok characterized linear embeddings $\varphi \colon Gr(d_1,m) \hookrightarrow Gr(d_2,n)$ when the differential $d \varphi$ of $\varphi$ preserves the decomposability of tangent vectors.

		\begin{theorem} [{\cite[Proposition 1, 3 and 4]{MO}}] \label{N. Mok's result}
			Let $\varphi \colon Gr(d_1,m) \hookrightarrow Gr(d_2,n)$ be a holomorphic embedding with $2 \le d_1 \le d_2$ and $2 \le m-d_1 \le n-d_2$.
			Assume that $d \varphi$ transforms decomposable tangent vectors into decomposable tangent vectors.
			Then either $\varphi$ is linear up to automorphisms of $Gr(d_1,m)$ or $Gr(d_2,n)$, or the image of $\varphi$ lies on some projective space in $Gr(d_2,n)$
			\textnormal{(}Here, $Y$ is a projective space in $Gr(d_2,n)$ if and only if $i(Y)$ is a projective space in $\Pbb^{\binom{n}{d_2}-1}$ where $i \colon Gr(d_2,n) \hookrightarrow \Pbb^{\binom{n}{d_2}-1}$ is the Pl\"{u}cker embedding\textnormal{)}.
		\end{theorem}
		
		Although N. Mok studied holomorphic embeddings between complex Grassmannians by the \emph{pushforward} of vector fields, there have been several approaches to study them by the \emph{pullback} of vector bundles.

		Consider more general situations: holomorphic maps from a compact complex manifold $Z$ into the complex Grassmannian $Gr(d,n)$.
		When we write $Gr(d,n) = Gr(d,\Cbb^n)$ definitely, the space $\Gamma (Gr(d, \Cbb^n), \ckE (d, \Cbb^n))$ of all holomorphic global sections of $\ckE (d, \Cbb^n)$ is naturally identified with $(\Cbb^n)^*$.
		%, that is, we assume that the total space of $\ckE (d,n)$ is contained in the total space of $(\Cbb^n)^* \otimes \mathcal{O}_{Gr(d,n)}$.

		Any holomorphic map $\psi \colon Z \to Gr(d,n)$ is determined completely by the pullback bundle $\psi^* (\ckE (d,n))$ because the fiber of $\psi^* (\ckE (d,n))$ at $z \in Z$ is equal to $(L_{\psi (z)})^* \subset (\Cbb^n)^* = \Gamma (Gr(d,n), \ckE (d,n))$.
		Let $\mathcal{E} := \psi^* (\ckE (d,n))$ and $\pi \colon (\Cbb^n)^* \otimes \mathcal{O}_Z \to \mathcal{E}$ be the pullback of the canonical surjective bundle morphism $(\Cbb^n)^* \otimes \mathcal{O}_{Gr(d,n)} \to \ckE (d,n)$ under the holomorphic map $\psi$, then the pair $(\mathcal{E}, \pi)$ satisfies that
		\begin{equation} \label{conditions for inducing holomorphic maps}
			\begin{tabular}{ll}
				$\bullet$ & $\mathcal{E}$ is a holomorphic vector bundle on $Z$ of rank $d$;\\
				$\bullet$ & $\pi \colon (\Cbb^n)^* \otimes \mathcal{O}_Z \to \mathcal{E}$ is a surjective holomorphic bundle\\
				& morphism.
				%$\bullet$ & $\mathcal{E}$ is a holomorphic vector bundle on $Z$ of rank $d$;\\
				%$\bullet$ & $\pi \colon (\Cbb^n)^* \otimes \mathcal{O}_Z \to \mathcal{E}$ is a surjective holomorphic bundle morphism.
			\end{tabular}
		\end{equation}

		Conversely, assume that a pair $(\mathcal{E}, \pi)$ satisfies \eqref{conditions for inducing holomorphic maps}.
		For each $z \in Z$, $\pi$ induces a surjective linear map $\pi_z \colon (\Cbb^n)^* \to \mathcal{E}_z$ between fibers at $z$.
		From these maps, define a holomorphic map $\psi \colon Z \to Gr(d,n)$ by the composition of holomorphic maps
		\[
			Z \quad \overset{\tilde{\pi}}{\to} \quad Gr(n-d, (\Cbb^n)^*) \quad \overset{\perp}{\to} \quad Gr(d,(\Cbb^n)^{**}) = Gr(d,n)
		\]
		where $L_{\tilde{\pi} (z)} := \ker (\pi_z) \subset (\Cbb^n)^*$ and $L_{\perp (y)} \subset \Cbb^n$ is defined by the annihilator of $L_y \subset (\Cbb^n)^*$.
		Let $\mathcal{F} := \psi^* (\ckE (d,n))$ and $\theta \colon (\Cbb^n)^* \otimes \mathcal{O}_Z \to \mathcal{F}$ be the pullback of the canonical surjective bundle morphism $(\Cbb^n)^* \otimes \mathcal{O}_{Gr(d,n)} \to \ckE (d,n)$ under the map $\psi$.
		Then we have $\mathcal{\check{F}} = \pi^* (\mathcal{\check{E}})$ where $\pi^* \colon \mathcal{\check{E}} \hookrightarrow \Cbb^n \otimes \mathcal{O}_Z$ is the pullback map of $\pi \colon (\Cbb^n)^* \otimes \mathcal{O}_Z \to \mathcal{E}$, thus $\pi^*$ gives an isomorphism $\mathcal{\check{E}} \simeq \mathcal{\check{F}}$ and $\pi^*$ is the composition of $\theta^* \colon  \mathcal{\check{F}} \hookrightarrow \Cbb^n \otimes \mathcal{O}_Z$ with this isomorphism.

		For more details on the correspondence between maps and vector bundles, see \cite[Section 23]{BT} (in the differential category), \cite[page 207--209]{GH} or \cite[Remark 4.3.21]{HUD}.
		The subsequent results are proved from this aspect.

		In \cite{FE}, S. Feder considered holomorphic embeddings $\varphi \colon \Pbb^m \hookrightarrow \Pbb^n$ and classified them by means of their degrees.
		Let $\mathcal{O}_{\Pbb^k} (1)$ be the holomorphic line bundle which corresponds to a hyperplane $H$ of $\Pbb^k$, then every holomorphic line bundle on $\Pbb^k$ can be expressed as $\mathcal{O}_{\Pbb^k} (r) := \mathcal{O}_{\Pbb^k} (1)^{\otimes r}$ for some $r \in \Zbb$ up to isomorphisms.
		Given a holomorphic map $\psi \colon \Pbb^m \to \Pbb^n$, the pullback bundle of $\mathcal{O}_{\Pbb^n} (1)$ under the map $\psi$ is isomorphic to $\mathcal{O}_{\Pbb^m} (r)$ for a unique integer $r$, and we call it the \emph{degree of $\varphi$}.
		A holomorphic embedding $\varphi$ is linear if and only if the degree of $\varphi$ is $1$.

		\begin{theorem} [{\cite[Theorem 1.2, 2.1 and 2.2]{FE}}] \label{S. Feder's result}
			\hskip 1.0cm
			\begin{enumerate}
				\item[(a)] Let $\varphi \colon \Pbb^m \hookrightarrow \Pbb^n$ be a holomorphic embedding. Then we have
				\[
					\text{the degree of } \varphi =\\
					\begin{cases}
						1, & \text{if } n < 2m\\
						1 \text{ or } 2, & \text{if } n = 2m
					\end{cases}
					\centermark{2}{.}
				\]
				\item[(b)] If $n > 2m$, then for any $r > 0$, there is a holomorphic embedding $\Pbb^m \hookrightarrow \Pbb^n$ of degree $r$.
			\end{enumerate}
		\end{theorem}

		In \cite{TA}, H. Tango considered holomorphic embeddings $\varphi \colon \Pbb^{n-2} \hookrightarrow Gr(2,n)$ with $n \ge 4$ and classified their images.
		In this case, $\varphi$ is linear if and only if the image of $\varphi$ equals $\{ x \in Gr(2,n) \ | \ p \in L_x \}$ for some $p \in \Cbb^n$.
		To state H. Tango's result, we need to define some subvarieties of $Gr(2,n)$ which are biholomorphic to $\Pbb^{n-2}$. For $x \in Gr(2,n)$, choose a basis $\{ v_1, v_2 \}$ of $L_x \subset \Cbb^n$, and construct the $2 \times n$ matrix of rank $2$ whose $i$\textsuperscript{th} row is the transpose of $v_i$ for $i=1$ and $2$.
		Although the choice of bases of $L_x$ is not unique, $\{ w_1, w_2 \}$ is a basis of $L_x$ if and only if the change of basis from $\{ v_1, v_2 \}$ to $\{ w_1, w_2 \}$ is an invertible $2 \times 2$ matrix.
		So we express an element in $Gr(2,n)$ as the equivalence class of a $2 \times n$ matrix of rank $2$
		\[
			\left[
			\begin{pmatrix}
				* & * & \cdots & * & *\\
				* & * & \cdots & * & *
			\end{pmatrix}
			\right]
		\]
		where the equivalence relation is given by
		\begin{center}
			$A \sim B$ \quad if and only if \quad $A = g \, B$ for some $g \in GL(2, \Cbb)$.
		\end{center}
		When $n \ge 4$, define subvarieties $X_{n-1,1}^0$ and $X_{n-1,1}^1$ of $Gr(2,n)$ by
		\begin{equation} \label{subvarieties 1 of Gr(2,n) which is biholomorphic to P^(n-2)}
			\begin{split}
				X_{n-1,1}^0
				&:= \left\{ \left[
				\begin{pmatrix}
					1 & 0 & \cdots & 0 & 0\\
					0 & x_0 & \cdots & x_{n-3} & x_{n-2}
				\end{pmatrix}
				\right] \ \Big| \ [{\bf x}] \in \Pbb^{n-2} \right\};\\
				X_{n-1,1}^1
				&:= \left\{ \left[
				\begin{pmatrix}
					x_0 & x_1 & \cdots & x_{n-2} & 0\\
					0 & x_0 & \cdots & x_{n-3} & x_{n-2}
				\end{pmatrix}
				\right] \ \Big| \ [{\bf x}] \in \Pbb^{n-2} \right\}
			\end{split}
		\end{equation}
		where $[{\bf x}] := [x_0 : x_1 : \cdots : x_{n-2}]$, and define subvarieties $\ckX_{3,1}^0$ and $\ckX_{3,1}^1$ of $Gr(2,4)$ by
		\begin{equation} \label{subvarieties 2 of Gr(2,n) which is biholomorphic to P^(n-2)}
			\begin{split}
				\ckX_{3,1}^0 &:= \phi (X_{3,1}^0);\\
				\ckX_{3,1}^1 &:= \phi (X_{3,1}^1)
			\end{split}
		\end{equation}
		where $\phi \colon Gr(2,4) \to Gr(2,4)$ is a dual map.
		For a quadric hypersurface $S$ of $\Pbb^4$, define a subvariety $X_q (S)$ of $Gr(2,5)$ by
		\begin{equation} \label{subvarieties 3 of Gr(2,n) which is biholomorphic to P^(n-2)}
			X_q (S) := \{ x \in Gr(2,5) \ | \ L_x \subset C(S) \}
		\end{equation}
		where $C(S) \subset \Cbb^5$ is the affine cone over $S$.

		\begin{theorem} [{\cite[Theorem 5.1 and 6.2]{TA}}] \label{H. Tango's result}
			Let $\varphi \colon \Pbb^{n-2} \hookrightarrow Gr(2,n)$ be a holomorphic embedding and $X$ be the image of $\varphi$.
			\begin{enumerate}
				\item[(a)] If $n = 4$, then $X \simeq X_{3,1}^0, \, X_{3,1}^1, \, \ckX_{3,1}^0$ or $\ckX_{3,1}^1$.
				\item[(b)] If $n = 5$, then $X \simeq X_{4,1}^0, \, X_{4,1}^1$ or $X_q (S)$ where $S$ is a fixed non-singular quadric hypersurface of $\Pbb^4$.
				\item[(c)] If $n \ge 6$, then $X \simeq X_{n-1,1}^0$ or $X_{n-1,1}^1$.
			\end{enumerate}
			\textnormal{(}Here, $X \simeq X_0$ if and only if $X = g \, X_0$ for some $g \in \PGL(n, \Cbb)$.\textnormal{)}
		\end{theorem}

		In \cite{SU}, J. C. Sierra and L. Ugaglia classified all the holomorphic embeddings $\varphi \colon \Pbb^m \hookrightarrow Gr(2,n)$ such that the compositions of the Pl\"{u}cker embedding $Gr(2,n) \hookrightarrow \Pbb^{\binom{n}{2}-1}$ with them are given by a linear system of quadrics in $\Pbb^m$.

		\begin{theorem}[{\cite[Theorem 2.12]{SU}}] \label{J. C. Sierra and L. Ugaglia's result}
			Let $\varphi \colon \Pbb^m \hookrightarrow Gr(2,n)$ be a holomorphic embedding satisfying that the line bundle $\wedge^2 \varphi^* (\ckE (2,n))$ is isomorphic to $\mathcal{O}_{\Pbb^m} (2)$.
			Let $E := \varphi^* (\ckE (2,n))$, then one of the following holds:
			\begin{enumerate}
				\item[(a)] $E \simeq \mathcal{O}_{\Pbb^m} \oplus \mathcal{O}_{\Pbb^m} (2)$.
				\item[(b)] $E \simeq \mathcal{O}_{\Pbb^m} (1) \oplus \mathcal{O}_{\Pbb^m} (1)$.
				\item[(c)] $m=3$ and $E = F \otimes \mathcal{O}_{\Pbb^3} (1)$, where $F$ is the kernel of the surjective bundle morphism $T_{\Pbb^3} \otimes \mathcal{O}_{\Pbb^3}(-1) \to \mathcal{O}_{\Pbb^3} (1)$ given as in \emph{\cite[Section 7]{BA1})}, has a resolution of the form:
				\[
					0 \ \to \ \mathcal{O}_{\Pbb^3} (-2) \ \to \ \bigoplus^4 \mathcal{O}_{\Pbb^3} (-1) \ \to \ \bigoplus^5 \mathcal{O}_{\Pbb^3} \ \to \ E \ \to \ 0.
				\]
				\item[(d)] $m=2$ and $E$ has a resolution of the form:
				\[
					0 \ \to \ \mathcal{O}_{\Pbb^2} (1) \ \to \ E \ \to \ m_P \otimes \mathcal{O}_{\Pbb^2} (1) \ \to \ 0
				\]
				where $m_P$ denotes the ideal sheaf of a point $P \in \Pbb^2$.
				\item[(e)] $m=2$ and $E$ has a resolution of the form:
				\[
					0 \ \to \ \bigoplus^2 \mathcal{O}_{\Pbb^2} (-1) \ \to \ \bigoplus^4 \mathcal{O}_{\Pbb^2} \ \to \ E \ \to \ 0.
				\]
			\end{enumerate}
		\end{theorem}

		Given a holomorphic vector bundle $\mathcal{F}$ on a compact complex manifold $Z$ of rank $d$, assume that $\mathcal{F}$ is generated by global sections, then the evaluation map $\Gamma (Z, \mathcal{F}) \otimes \mathcal{O}_Z \to \mathcal{F}$ is surjective.
		After composing with a linear isomorphism $(\Cbb^n)^* \simeq \Gamma (Z, \mathcal{F})$, we have a surjection $\pi \colon (\Cbb^n)^* \otimes \mathcal{O}_Z \to \mathcal{F}$.
		The pair $(\mathcal{F}, \pi)$ induces a holomorphic map $\psi \colon Z \to Gr(d,n)$ (uniquely up to linear automorphisms of $Gr(d,n)$).

		\begin{example}[{\cite[Example 1.5$-$1.9 and Remark 3.5]{SU}}] \label{examples of double Veronese embeddings}
			For each item, let $\mathcal{E}$ be the holomorphic vector bundle on $\Pbb^m$ in the same item of Theorem \ref{J. C. Sierra and L. Ugaglia's result}, and let $\pi \colon (\Cbb^n)^* \otimes \mathcal{O}_{\Pbb^m} \to \mathcal{E}$ be a composition of the evaluation map of $\Gamma (\Pbb^m, \mathcal{E})$ with a linear isomorphism $(\Cbb^n)^* \simeq \Gamma (\Pbb^m, \mathcal{E})$.
			\begin{enumerate}
				\item[(a)] The pair $(\mathcal{E}, \pi)$ induces a holomorphic embedding
				$\varphi \colon \Pbb^m \hookrightarrow Gr(2,n)$, where $n := \binom{m+2}{2} + 1$, which is given by the family of ruling lines of a cone over the second Veronese embedding $\nu_2 (\Pbb^m) \subset \Pbb^{n-2}$ with vertex a point.

				\item[(b)] The pair $(\mathcal{E}, \pi)$ induces a holomorphic embedding
				$\varphi \colon \Pbb^m \hookrightarrow Gr(2,2m+2)$ which is given by the family of lines joining the corresponding points on two disjoint $\Pbb^m$'s in $Gr(2,2m+2)$.
				Moreover, the holomorphic embedding $\varphi_0 \colon \Pbb^m \hookrightarrow Gr(2,m+2)$ whose image equals the subvariety $X_{m+1,1}^1$, which is given as in \eqref{subvarieties 1 of Gr(2,n) which is biholomorphic to P^(n-2)}, can be obtained by projecting from $\varphi (\Pbb^m) \subset Gr(2,2m+2)$ to $Gr(2,m+2)$, and the pullback bundle $\varphi_0^* (\ckE (2,m+2))$ is also isomorphic to $\mathcal{E}$.
				%, and the holomorphic vector bundles which correspond to the embeddings $\Pbb^m \hookrightarrow Gr(2,2m+2)$ and $\Pbb^m \hookrightarrow Gr(2,m+2)$ are isomorphic.

				\item[(c)] The pair $(\mathcal{E}, \pi)$ induces a holomorphic embedding
				$\varphi \colon \Pbb^3 \hookrightarrow Gr(2,5)$ whose image equals the subvariety $X_q (S)$, which is given as in \eqref{subvarieties 3 of Gr(2,n) which is biholomorphic to P^(n-2)}.

				\item[(d)] The pair $(\mathcal{E}, \pi)$ induces a holomorphic embedding
				$\varphi \colon \Pbb^2 \hookrightarrow Gr(2,5)$ which is a composition of the holomorphic embedding $\Pbb^3 \hookrightarrow Gr(2,5)$ in (c) with a linear embedding $\Pbb^2 \hookrightarrow \Pbb^3$.

				\item[(e)] The pair $(\mathcal{E}, \pi)$ induces a holomorphic embedding
				$\varphi \colon \Pbb^2 \hookrightarrow Gr(2,4)$ whose image equals the subvariety $\ckX_{3,1}^1$, which is given as in \eqref{subvarieties 2 of Gr(2,n) which is biholomorphic to P^(n-2)}.
			\end{enumerate}
		\end{example}

		Similarly, in \cite{HUS}, S. Huh classified all the holomorphic embeddings $\Pbb^m \hookrightarrow Gr(2,n)$ such that the compositions of the Pl\"{u}cker embedding $Gr(2,n) \hookrightarrow \Pbb^{\binom{n}{2}-1}$ with them are given by a linear system of cubics in $\Pbb^m$.

		Motivated by the previous results, we consider holomorphic embeddings $\varphi \colon Gr(2,m) \hookrightarrow Gr(2,n)$ and obtain the following numerical conditions on $m$ and $n$ for the linearity of $\varphi$:

		% 페이지 넘기기
		\newpage

		\begin{namedthm*}{Main Theorem} \hypertarget{main theorem}{}
			Let $\varphi \colon Gr(2,m) \hookrightarrow Gr(2,n)$ be a holomorphic embedding.
			\begin{enumerate}
				\item[(a)] If $9 \le m$ and $n \le \frac{3m-6}{2}$, then $\varphi$ is linear.
				\item[(b)] If $4 \le m$ and $n = m+1$, then either $\varphi$ is linear, or $m=4$ and $\varphi$ is a composition of a linear holomorphic embedding of $Gr(2,4)$ into $Gr(2,5)$ with a dual map $\phi \colon Gr(2,4) \to Gr(2,4)$.
			\end{enumerate}
		\end{namedthm*}

		\hyperlink{main theorem}{Main Theorem} follows from Theorem \ref{main theorem : general case} and \ref{main theorem : special case}.
		We do not have enough examples of non-linear embeddings $\varphi \colon Gr(2,m) \hookrightarrow Gr(2,n)$ except when 
		$m \ge 3$ and $n > m(m-1)$ (see Example \ref{examples of non-linear embeddings} (b)).
		Since $m(m-1)$ is much greater than both $\frac{3m-6}{2}$ and $m+1$, the assumptions in \hyperlink{main theorem}{Main Theorem} can be improved. To find a sharp condition of $m$ and $n$ for the linearity will be an interesting problem.

		Although S. Feder and H. Tango dealt with different cases, they used a similar numerical technique.
		The method can be applied to every holomorphic embedding $\varphi \colon Gr(d_1,m) \hookrightarrow Gr(d_2,n)$ if $d_1, \, d_2$ are fixed and $d_2 (n-d_2) \le 2 d_1 (m-d_1)$.

		\begin{enumerate}
			\item[Step 1.]\hypertarget{step 1}{} Let $E$ be the pullback bundle of $\ckE (d_2,n)$ under the embedding $\varphi$.
			If $\varphi$ is linear, then $E$ is isomorphic to $\ckE (d_1,m) \oplus ( \bigoplus^{d_2-d_1} \mathcal{O}_{Gr(d_1,m)} )$, thus the first Chern class $c_1 (E)$ of $E$ equals $c_1 (\ckE (d_1,m))$.
			Conversely, if $c_1 (E) = c_1 (\ckE (d_1,m))$, then $d \varphi$ preserves the decomposability of tangent vectors (see Remark \ref{application of N. Mok's result}), thus either $\varphi$ is linear up to automorphisms of $Gr(d_1,m)$ or $Gr(d_2,n)$, or $\varphi$ embeds $Gr(d_1,m)$ into some projective space in $Gr(d_2,n)$ by Theorem \ref{N. Mok's result}.
			To distinguish the linear case from the others, we need additional conditions.

			\item[Step 2.]\hypertarget{step 2}{} Choose a $\Zbb$-module basis $\mathcal{B}$ of the cohomology ring of $Gr(d_1,m)$.
			The total Chern class of $E$ can be written uniquely as a linear combination of elements in $\mathcal{B}$ with coefficients $a,b, \cdots , c$ in $\Zbb$ (Here, $a$ is determined so that $c_1 (E) = a \, c_1 (\ckE (d_1,m))$).
			Let $N$ be the pullback bundle of the normal bundle of $\varphi (Gr(d_1,m))$ in $Gr(d_2,n)$ under the embedding $\varphi$.
			Using canonical short exact sequences for $E(d_1,m)$ and for $E(d_2,n)$, we construct an equation of the total Chern class of $N$ in terms of the total Chern classes of $E, \ E(d_1,m)$, their dual bundles and tensor product bundles.
			Thus each Chern class of $N$ can be written as a linear combination of elements in $\mathcal{B}$ with coefficients in the multivariate polynomial ring $\Zbb [a,b, \cdots , c]$ over $\Zbb$.
			The Euler class of $N_{\Rbb}$, which is the real vector bundle corresponding to $N$, equals the pullback bundle of the Poincar\'{e} dual to the homology class of $\varphi (Gr(d_1,m))$ under the embedding $\varphi$.
			So we also express the Euler class of $N_{\Rbb}$ as a linear combination of elements in $\mathcal{B}$ with coefficients in $\Zbb [a,b, \cdots , c]$.
			By definitions of Chern classes, the top Chern class of $N$ equals the Euler class of $N_{\Rbb}$, and the $k$\textsuperscript{th} Chern class of $N$ equals $0$ if $k > \rank (N)$.
			Since $a,b, \cdots , c$ are integers, these equations are Diophantine equations (If $\rank (N) > \dim (Gr(d_1,m))$ or, equivalently, $d_2 (n-d_2) > 2d_1 (m-d_1)$, then we cannot obtain any equation).

			\item[Step 3.]\hypertarget{step 3}{} In general, it is hard to solve these kinds of equations.
			To overcome this difficulty, we need additional conditions on $a,b, \cdots c$, such as inequalities.
			Applying a criterion of the numerical non-negativity of Chern classes of holomorphic vector bundles to suitable holomorphic vector bundles on $Gr(d_1,m)$, we obtain some inequalities in $a,b, \cdots , c$.
			To obtain other conditions, we have to look for a useful method case by case.
		\end{enumerate}
		
		In this way, we can think of the problem on the classification of holomorphic embeddings between complex Grassmannians as the problem on solving the obtained Diophantine equations and inequalities.
		%By the weak equivalent condition for the linearity of $\varphi$ in \hyperlink{step 1}{Step 1}, we characterize the linearity of $\varphi$ from $a,\, b$ and $c$.
		We prove \hyperlink{main theorem}{Main Theorem} by applying the numerical technique together with further results to our case.

		%The {paper} consists of four {section}s.

		In {Section} \ref{sec2}, we introduce backgrounds about complex Grassmannians $Gr(d,m)$, such as Schubert cycles, the universal and the universal quotient bundles.
		While most subjects in this {section} are basic and well-known, there are two remarkable subjects which play significant roles in reaching our goal.
		First, we provide two $\Zbb$-module bases of the cohomology ring of $Gr(2,m)$.
		One is the set of all Schubert cycles on $Gr(2,m)$ and the other is the set of all monomials of the form $(c_1(\ckE (2,m)))^i \, (c_2(\ckE (2,m)))^j$ satisfying a certain condition on $i$ and $j$ (Proposition \hyperlink{new basis and relation between two bases : target}{\ref{new basis and relation between two bases}}).
		As we already mentioned, the formal basis arises from a cell decomposition of $Gr(2,m)$, thus it is useful to verify geometric features of $Gr(2,m)$.
		The latter basis, denoted by $\mathcal{C}$, arises from a ring generator $\{ c_1 (\ckE (2,m)), c_2 (\ckE (2,m)) \}$ of the cohomology ring of $Gr(2,m)$, thus it is useful to express multiplications of cohomology classes until the degree is not greater than $2(m-2)$.
		For this reason, we use the basis $\mathcal{C}$ to express cohomology classes as linear combinations like $\mathcal{B}$ in \hyperlink{step 2}{Step 2}.
		Second, we provide W. Barth and A. Van de Ven's results on the decomposability of holomorphic vector bundles on complex Grassmannians of rank $2$ (Proposition \ref{result on holomorphic vector bundles on P^k of rank 2} and \ref{result on holomorphic vector bundles on Gr(d,m) of rank 2}).
		If a holomorphic vector bundle $\mathcal{E}$ on $Gr(2,m)$ of rank $2$ satisfies the assumptions of their results, then we can handle $\mathcal{E}$ easily.

		In {Section} \ref{sec3}, we consider holomorphic embeddings $\varphi \colon Gr(2,m) \hookrightarrow Gr(2,n)$ and their linearity.
		As in \hyperlink{step 2}{Step 2}, we set the integral coefficients $a,b$ and $c$ to express the total Chern class of $E$ with respect to the basis $\mathcal{C}$, and provide an equivalent condition on the pair $(a,b,c)$ for the linearity of $\varphi$ (Proposition \ref{equivalent conditions for the linearity of an embedding}).
		When we focus on the coefficients of the powers of $c_1(\ckE (2,m))$ (resp. $c_2(\ckE (2,m))$) in the equation of the total Chern class of $N$ in \hyperlink{step 2}{Step 2}, we derive a refined equation whose both sides are polynomials in one variable $c_1(\ckE (2,m))$ (resp. $c_2(\ckE (2,m))$).
		If $\rank (N) = 2n-2m$ is not greater than $m-2$, then these two refined equations preserve the coefficients of the cohomology classes of degree $2(2n-2m)$ (Proposition \ref{refined total Chern class of N}).
		Solving the refined equations and the equation of the Euler class of $N_{\Rbb}$ in \hyperlink{step 2}{Step 2} together with the inequalities in \hyperlink{step 3}{Step 3}, we obtain a lower bound of the coefficient of $(c_1(\ckE (2,m)))^{2n-2m}$ in the top Chern class of $N$ with respect to $\mathcal{C}$ (Lemma \ref{lower bound of alpha_(2n-2m)}) with an inequality in $a$ and $b$ (Proposition \ref{inequality : a^2 > 4b}).

		In {Section} \ref{sec4}, we prove \hyperlink{main theorem}{Main Theorem} (a) and (b) separately.
		For the proof of (a), we first obtain an upper bound of $a$ (Proposition \ref{upper bound of sqrt(a^2-4b) and a} (b)) from all the previous results. This bound enables us to apply W. Barth and A. Van de Ven's results to $E$, thus we can solve the refined equation in $c_1 (\ckE (2,m))$ more easily (Theorem \ref{main theorem : general case}).
		For the proof of (b), we solve the equality of the top Chern class of $N$ with the Euler class of $N_{\Rbb}$ directly (Theorem \ref{main theorem : special case}). It is reasonable because $\rank (N) = 2$ is sufficiently small.

		Throughout the {paper}, a Grassmannian means a complex Grassmannian, a map means a holomorphic map, and a vector bundle means a holomorphic vector bundle by abuse of terminology.

	\vskip 0.5cm
	
	\section{Preliminaries} \label{sec2}
		
		We introduce here basic concepts about Schubert cycles, the universal and the universal quotient bundle on $Gr(d,m)$, together with W. Barth and A. Van de Ven's results which are about the decomposability of vector bundles on Grassmannians of rank $2$.
		For more details on Schubert cycles, see \cite{AR} or \cite[Section 1.5]{GH}, for more details on the universal quotient bundle, see \cite{TA}, and for more details on W. Barth and A. Van de Ven's results, see \cite{BV1} or \cite{BV2}.

		% Comment
		% We introduce here basic concepts about Grassmannians $Gr(d,m)$.
		% In {Subsection} \ref{subsec2.1}, we provide Schubert varieties, Schubert cycles on $Gr(d,m)$ and Pieri's formula which describes the multiplications of Schubert cycles on $Gr(d,m)$.
		% In general, the set of all Schubert cycles on $Gr(d,m)$ is a $\Zbb$-module basis of the cohomology ring of $Gr(d,m)$.
		% When $d=2$, we provide another $\Zbb$-module basis of the cohomology ring of $Gr(2,m)$, which is motivated by its ring generator, and the relation between these two bases.
		% In {Subsection} \ref{subsec2.2}, we provide the universal bundle, the universal quotient bundle on $Gr(d,m)$ and their total Chern classes.
		% In addition, we provide W. Barth and A. Van de Ven's results which are about the decomposability of vector bundles on Grassmannians of rank $2$.
		% For more details on {Subsection} \ref{subsec2.1}, see \cite{AR} and \cite[Section 1.5]{GH}, and for more details on {Subsection} \ref{subsec2.2}, see \cite{TA}, \cite{BV1} and \cite{BV2}.

		\subsection{Schubert cycles on Grassmannians} \label{subsec2.1}

			For a partial flag $0 \subset A_1 \subsetneq \cdots \subsetneq A_d \subset \Cbb^m$, let $\omega (A_1, \cdots , A_d)$ be the subvariety of $Gr(d,m)$ which is given by
			\[
				\{ x \in Gr(d,m) \ | \ \dim (L_x \cap A_i) \ge i \text{ for all } 1 \le i \le d \}.
			\]
			We call such a subvariety $\omega (A_1, \cdots , A_d)$ a \emph{Schubert variety of type $(a_1, \cdots , a_d)$} where $a_i := m-d+i- \dim (A_i)$ for $1 \le i \le d$. The (complex) codimension of $\omega (A_1, \cdots , A_d)$ in $Gr(d,m)$ is $\sum_{i=1}^{d} a_i$. We sometimes denote $(a_1, \cdots , a_d)$ simply by the bold lowercase letter $\bds{a}$.

			\begin{example} \label{examples of Schubert varieties of Gr(d,m)}
				There are some familiar Schubert varieties on $Gr(d,m)$, which are sub-Grassmannians of $Gr(d,m)$. Given a type $\bds{\star}$, let $X_{\star} := \omega (A_1, \cdots , A_d)$ be a Schubert variety of type $\bds{\star}$.
				\begin{enumerate}
					\item[(a)] $\bds{a} = (m-d, \cdots , m-d, 0)$ :
					Since $\dim (A_i) = i$ for all $1 \le i \le d-1$ and $\dim (A_d) = m$,
					\begin{align*}
						A_i &= \sspan (\{ v_1, \cdots , v_i \}) \quad \text{for all } 1 \le i \le d-1;\\
						A_d &= \Cbb^m
					\end{align*}
					for some linearly independent vectors $v_1, \cdots , v_{d-1} \in \Cbb^m$. So we have
					\begin{align*}
						X_{\bds{a}}
						&= \{ x \in Gr(d,m) \ | \ A_{d-1} \subset L_x \}\\
						&\simeq \Pbb \left( \quotient{\Cbb^m}{A_{d-1}} \right) \simeq \Pbb^{m-d}
					\end{align*}
					(When $d=2$, $X_{\bds{a}} = X_{n-1,1}^0$ up to linear automorphisms of $Gr(2,n)$, where $X_{n-1,1}^0$ is given as in \eqref{subvarieties 1 of Gr(2,n) which is biholomorphic to P^(n-2)}).

					$X_{\bds{a}}$ is a maximal projective space in $Gr(d,m)$, that is, there is not a projective space in $Gr(d,m)$ containing it properly.

					\item[(b)] $\bds{b} = (\underbrace{m-d, \cdots , m-d}_k, 0, \cdots , 0)$ :
					Since $\dim (A_i) = i$ for all $1 \le i \le k$ and $\dim (A_j) = m-d+j$ for all $k+1 \le j \le d$,
					\begin{align*}
						X_{\bds{b}}
						&= \{ x \in Gr(d,m) \ | \ A_k \subset L_x \}\\
						&\simeq Gr(d-k, \quotient{\Cbb^m}{A_k}) \simeq Gr(d-k,m-k).
					\end{align*}

					\item[(c)] $\bds{c} = (k, \cdots , k)$ :
					Since $\dim (A_i) = m-d+i-k$ for all $1 \le i \le d$,
					\begin{align*}
						X_{\bds{c}}
						&= \{ x \in Gr(d,m) \ | \ L_x \subset A_d \}\\
						&= Gr(d,A_d) \simeq Gr(d,m-k).
					\end{align*}
					So any subvariety $Gr(d,H)$ of $Gr(d,m)$ where $H$ is a subspace of $\Cbb^m$ is of this form.

					\item[(d)] $\bds{d} = (\underbrace{m-d, \cdots , m-d}_k, l, \cdots , l)$ :
					Combining the results of (b) and (c),
					\[
						X_{\bds{d}} \simeq Gr(d-k, m-k-l),
					\]
					which is contained in a Schubert variety of the form $X_{\bds{b}}$. The inclusion $X_{\bds{d}} \subset X_{\bds{b}}$ corresponds to the inclusion $X_{\bds{c}} \subset Gr(d,m)$ in (c).
				\end{enumerate}
			\end{example}

			Two Schubert varieties of types $\bds{a}$ and $\bds{b}$ have the same homology class if and only if $\bds{a} = \bds{b}$. We denote the Poincar\'{e} dual to a Schubert variety of type $(a_1, \cdots , a_d)$ by $\omega_{a_1, \cdots , a_d}$ and call it the \emph{Schubert cycle of type $(a_1, \cdots , a_d)$}. Since the codimension of a Schubert variety of type $(a_1, \cdots , a_d)$ is $\sum_{i=1}^d a_i$,
			\[
				\omega_{a_1, \cdots , a_d} \in H^{2 \left( \sum_{i=1}^d a_i \right)} (Gr(d,m), \Zbb)
			\]
			and the set of all Schubert cycles describes every cohomology group of $Gr(d,m)$ completely as follows:
			\[
				H^i (Gr(d,m), \Zbb) =
				\begin{cases}
					0, & \text{if } i \text{ is odd}\\
					\sspan (\mathcal{B}_k), & \text{if } i \, (=2k) \text{ is even}
				\end{cases}
			\]
			where $\mathcal{B}_k$ is a basis which is given by
			\[
				\left\{ \omega_{a_1, \cdots, a_d} \ \Big| \ m-d \ge a_1 \ge \cdots \ge a_d \ge 0; \ \sum_{i=1}^d a_i = k \right\}.
			\]
			In particular, when $d=2$,
			\begin{equation} \label{standard basis of cohomology group}
				\{ \omega_{k-i,i} \ | \ m-2 \ge k-i \ge i \ge 0 \}
			\end{equation}
			is a basis of $\cohomgr{m}{2k}$. For $k=2m-4$, $\cohomgr{m}{2(2m-4)} \simeq \Zbb$ is generated by $\omega_{m-2, m-2} = \omgaa^{m-2}$. Every $\Gamma \in \cohomgr{m}{2(2m-4)}$ is of the form $c_{\Gamma} \, \omgaa^{m-2}$ for some integer $c_{\Gamma}$, thus we identify $\Gamma$ with $c_{\Gamma} \in \Zbb$.

			By Example \ref{examples of Schubert varieties of Gr(d,m)} (c), the Poincar\'{e} dual to the homology class of the subvariety $Gr(d,H) \subset Gr(d,m)$ where $H$ is a subspace of $\Cbb^m$ is $\omega_{k, \cdots , k}$ for some $0 \le k \le m-d$. The next proposition is about its converse.

			\begin{proposition} [{\cite[Theorem 7 and Corollary 5]{WA} or \cite[Example 11]{BR}}] \label{rigidity of sub-Grassmannians}
				For $m \ge d \ge 2$, let $X_0$ be a subvariety of $Gr(d,m)$ satisfying that the Poincar\'{e} dual to the homology class of $X_0$ is $\omega_{k, \cdots , k}$ for some $0 \le k \le m-d$. Then $X_0 = Gr(d, H)$ for some $(m-k)$-dimensional subspace $H$ of $\Cbb^m$.
			\end{proposition}
			%
			%		\begin{proof}
			%			Since the codimension of $X_0$ in $Gr(d,m)$ is $dk$, if $k=0$ (resp. $k = m-d$), then $X_0 = Gr(d,m)$ (resp. $X_0 = \{ \text{point} \}$). So we may assume that $0 < k < m-d$. Note that the Young diagram of $X_0$ is a rectangle with height $d \ge 2$ and width $m-d-k \ge 1$. By Example \ref{examples of Schubert varieties of Gr(d,m)} (c), a Schubert variety of type $(k, \cdots , k)$ is a sub-Grassmannian, thus is smooth. Applying \cite[Theorem 7 and Corollary 5]{WA} (or \cite[Example 11]{BR}), the only solutions of the Schubert differential system of type $(k, \cdots , k)$ are Schubert varieties of type $(k, \cdots , k)$. Hence, $X_0$ is a Schubert variety of type $(k, \cdots , k)$, that is, $X_0 \simeq Gr(d,m-k)$ is a sub-Grassmannian of $Gr(d,m)$ by Example \ref{examples of Schubert varieties of Gr(d,m)} (c), again.
			%		\end{proof}
			%
			%		In the proof of Proposition \ref{rigidity of sub-Grassmannians}, we show that every subvariety of $Gr(d,m)$ which corresponds to $\omega_{k, \cdots , k}$ is a Schubert variety of type $(k, \cdots , k)$. We call such a property the \emph{rigidity}. But it is not true that the rigidity holds for every type. In \cite[Theorem 7 and Corollary 5]{WA}, and \cite[Theorem 1.2]{HM2}, types $(a_1, \cdots , a_d)$ for which the rigidity holds are classified.

			The multiplications of Schubert cycles are commutative and satisfy the following rule, named Pieri's formula.

			\begin{lemma} [Pieri's formula] \label{Pieri's formula}
				In $Gr(d,m)$, for $m-d \ge a_1 \ge a_2 \ge \cdots \ge a_d \ge 0$ and $m-d \ge h \ge 0$,
				\begin{equation} \label{multiplication of Pieri's formula}
					\omega_{a_1, a_2, \cdots , a_d} \, \omega_{h, 0, \cdots , 0} = \sum_{(b_1, b_2, \cdots , b_d) \in I} \omega_{b_1, b_2, \cdots , b_d}
				\end{equation}
				where $I$ is the set of all pairs $(b_1, b_2, \cdots , b_d) \in \Zbb^d$ satisfying
				\begin{gather*}
					m-d \ge b_1 \ge a_1 \ge b_2 \ge a_2 \ge \cdots \ge b_d \ge a_d \ge 0;\\
					\left( \sum_{i=1}^{d} a_i \right) + h = \sum_{i=1}^{d} b_i.
				\end{gather*}
			\end{lemma}

			For the proof of Lemma \ref{Pieri's formula}, see \cite[page 203]{GH}. To multiply two general Schubert cycles on $Gr(d,m)$ by using Lemma \ref{Pieri's formula}, we need to express this multiplication as a composition of finite multiplications of the form \eqref{multiplication of Pieri's formula}. In general, it is not easy. But when $d=2$, we have a refined Pieri's formula which enables us to multiply any two general Schubert cycles easily. Before describing this formula, we adopt the following convention:

			\begin{convention} \label{convention for Schubert cycles}
				In $Gr(2,m)$, let $\omega_{k,l} = 0$ unless $m-2 \ge k \ge l \ge 0$.
			\end{convention}

			From now on, we always assume Convention \ref{convention for Schubert cycles} when $d=2$.

			\begin{corollary} [Refined Pieri's formula] \label{refined Pieri's formula}
				Schubert cycles on $Gr(2,m)$ satisfy the following relations:
				\begin{enumerate}
					\item[(a)] {\rm(\cite[Lemma 4.2 (\romannum{1})]{TA})} \ $\omega_{i,j} \, \omgaa = \omega_{i+1,j+1}$.
					\item[(b)] {\rm(Restate of Lemma \ref{Pieri's formula})} \ $\omgko{i} \, \omgko{j} = \omega_{i+j,0} + \omega_{i+j-1,1} + \cdots + \omega_{i+1,j-1} + \omega_{i,j}$.
				\end{enumerate}
			\end{corollary}

			Using Corollary \ref{refined Pieri's formula} and the commutativity of multiplications, we can multiply Schubert cycles on $Gr(2,m)$ easily. For example,
			\begin{align*}
				\omega_{8,5} \, \omega_{7,3}
				&= (\omega_{3,0} \, \omega_{1,1}^5) \, (\omega_{4,0} \, \omega_{1,1}^3) = \omega_{4,0} \, \omega_{3,0} \, \omega_{1,1}^8\\
				&= (\omega_{7,0} + \omega_{6,1} + \omega_{5,2} + \omega_{4,3}) \, \omega_{1,1}^8\\
				&= \omega_{15,8} + \omega_{14,9} + \omega_{13,10} + \omega_{12,11}
			\end{align*}
			(Some terms can be omitted if $m<17$).

			Furthermore, using Corollary \ref{refined Pieri's formula}, we obtain the result on the multiplications of two Schubert cycles of complementary degrees.

			\begin{corollary} [{\cite[Lemma 4.2 (\romannum{2})]{TA}}] \label{dual Schubert cycle of a Schubert cycle}
				In $Gr(2,m)$, let $i, \, j, \, k$ and $l$ be integers with $m-2 \ge i \ge j \ge 0, \ m-2 \ge k \ge l \ge 0$ and $i+j+k+l = 2m-4$. Then we have
				\[
					\omega_{i,j} \, \omega_{k,l} =
					\begin{cases}
						1, & \text{if } i+l = m-2 = j+k\\
						0, & \text{otherwise}
					\end{cases}
					\centermark{2}{.}
				\]
			\end{corollary}

			By Corollary \ref{dual Schubert cycle of a Schubert cycle}, for each $0 \le p \le 2m-4$, there is a bijection
			\begin{align*}
				& \tau_p \colon \{ \omega_{i,j} \ | \ m-2 \ge i \ge j \ge 0; \ i+j = p \}\\
				& \hskip 1.0cm \to \{ \omega_{k,l} \ | \ m-2 \ge k \ge l \ge 0; \ k+l = 2m-4-p \}
			\end{align*}
			which is defined by the property: the multiplication of $\omega_{i,j}$ and $\tau_p (\omega_{i,j})$ is equal to $1$. We call the image $\tau_p (\omega_{i,j}) = \omega_{m-2-j, m-2-i}$ the \emph{dual Schubert cycle of $\omega_{i,j}$}.

			Note that the cohomology ring of $Gr(2,m)$
			\[
				\cohomgr{m}{\bullet} = \bigoplus_{k=0}^{2m-4} \cohomgr{m}{2k}
			\]
			is generated by $\omgao$ and $\omgaa$ as a ring. Motivated by this fact, we find a new basis of $\cohomgr{m}{2k}$ whose elements are expressed by $\omgao$ and $\omgaa$.

			\begin{proposition} \label{new basis and relation between two bases} \hypertarget{new basis and relation between two bases : target}{}
				\begin{enumerate}
					\item[(a)] For $0 \le k \le 2m-4$, the set of Schubert cycles
					\begin{equation} \label{new basis of cohomology group of any degree}
						%& \left\{ \omgao^i \, \omgaa^j \ | \ m-2 \ge i+j \ge i \ge 0; \ i+2j = k %\right\}\\
						%=&
						\left\{ \omgao^{k-2i} \, \omgaa^{i} \ | \ m-2 \ge k-i \ge i \ge 0 \right\}
					\end{equation}
					forms a basis of $\cohomgr{m}{2k}$. In particular, when $0 \le k \le m-2$, the set of Schubert cycles
					\begin{equation} \label{new basis of cohomology group of low degree}
						\left\{ \omgao^{k-2i} \, \omgaa^{i} \ \Big| \ 0 \le i \le \left\lfloor \frac{k}{2} \right\rfloor \right\}
					\end{equation}
					forms a basis of $\cohomgr{m}{2k}$ where $\lfloor \bullet \rfloor$ is the maximal integer which does not exceed $\bullet$.
					\item[(b)] For $0 \le k \le m-2$, let $\Gamma \in \cohomgr{m}{2k}$ be a cohomology class. Then the coefficient of $\omega_{k,0}$ in $\Gamma$ with respect to the basis \eqref{standard basis of cohomology group} coincides with that of $\omgao^k$ in $\Gamma$ with respect to the basis \eqref{new basis of cohomology group of low degree}.
				\end{enumerate}
			\end{proposition}

			\begin{proof}
				(a)
				%First, the equality of \eqref{new basis of cohomology group of any degree} is clear.
				Using Corollary \ref{refined Pieri's formula}, for each $i \in \Zbb$ with $m-2 \ge k-i \ge i \ge 0$,
				\begin{eqnma} \label{new basis : equation 2}
					\omgao^{k-2i} \, \omgaa^i
					&= \left( \omgko{k-2i} + \sum_{j=1}^h a_{i,j} \, \omega_{k-2i-j,j} \right) \, \omgaa^i\\
					&= \omega_{k-i,i} + \sum_{j=1}^h a_{i,j} \, \omega_{k-i-j,i+j}
				\end{eqnma}
				for some non-negative integers $a_{i,j}$ and $h := \left\lfloor \frac{k-2i}{2} \right\rfloor$. Since $m-2 \ge k-i \ge i \ge 0$, the leading term $\omega_{k-i,i}$ of \eqref{new basis : equation 2} is not a zero.

				Note that $m-2 \ge k-i \ge i \ge 0$ if and only if $i_0 \le i \le \left\lfloor \frac{k}{2} \right\rfloor$ where $i_0 := \max \{ 0, k-m+2 \}$, so we have
				\begin{equation} \label{another expression of standard basis of cohomology group}
					\eqref{standard basis of cohomology group} = \left\{ \omega_{k-i,i} \ | \ i_0 \le i \le \left\lfloor \frac{k}{2} \right\rfloor \right\}.
				\end{equation}
				Let $\mathcal{A}_i := \sspan \left\{ \omgao^{k-2j} \, \omgaa^j \ | \ i \le j \le \left\lfloor \frac{k}{2} \right\rfloor \right\}$ for $i_0 \le i \le \lfloor \frac{k}{2} \rfloor$ and $\mathcal{A}_{\lfloor \frac{k}{2} \rfloor + 1} := 0$ (zero $\Zbb$-module).
				Then we have by \eqref{new basis : equation 2},
				\begin{equation} \label{new basis : equation 3}
					\mathcal{A}_{i+1} \subset \mathcal{A}_i \setminus \{ \omgao^{k-2i} \, \omgaa^i \}; \qquad
					\omega_{k-i,i} \in \omgao^{k-2i} \, \omgaa^i + \mathcal{A}_{i+1} \subset \mathcal{A}_i
				\end{equation}%
				for all $i_0 \le i \le \lfloor \frac{k}{2} \rfloor$.
				So the basis \eqref{another expression of standard basis of cohomology group} is contained in $\mathcal{A}_{i_0}$, which is the $\Zbb$-submodule generated by \eqref{new basis of cohomology group of any degree}.
				Furthermore, the basis \eqref{another expression of standard basis of cohomology group} and the set \eqref{new basis of cohomology group of any degree} have the same number of elements. Hence, \eqref{new basis of cohomology group of any degree} is also a basis of $\cohomgr{m}{2k}$.

				(b) Since $k \le m-2$, $i_0 = 0$.
				By \eqref{new basis : equation 3}, we have $\omega_{k,0} \in \omgao^k + \mathcal{A}_1$, $c \, \omgao^k \notin \mathcal{A}_1$ for all $c \in \Zbb$ and $\omega_{k-i,i} \in \mathcal{A}_i \subset \mathcal{A}_1$ for all $1 \le i \le \lfloor \frac{k}{2} \rfloor$.
				Hence, the coefficient of $\omega_{k,0}$ in $\Gamma$ with respect to the basis \eqref{standard basis of cohomology group} is equal to that of $\omgao^k$ in $\Gamma$ with respect to \eqref{new basis of cohomology group of low degree}.
			\end{proof}

			\begin{remark} \label{reason why 2n-2m <= m-2 is important}
				By Proposition \hyperlink{new basis and relation between two bases : target}{\ref{new basis and relation between two bases}} (a), the set of Schubert cycles
				\[
					\left\{ \omgao^{k-2i} \, \omgaa^i \ \Big| \ 0 \le i \le \left\lfloor \frac{k}{2} \right\rfloor \right\} = \left\{ \omgao^k, \ \omgao^{k-2} \, \omgaa, \cdots , \omgao^{k-2 \lfloor k/2 \rfloor} \, \omgaa^{\lfloor k/2 \rfloor} \right\}
				\]
				is linearly independent if $0 \le k \le m-2$, but it is linearly dependent if $m-1 \le k \le 2m-4$.
				For this reason, we assume that $(\rank (N) = ) \, 2n-2m \le m-2$ in Proposition \ref{refined total Chern class of N} where $N$ is the vector bundle on $Gr(2,m)$ which is given as in the introductory part of {Subsection} \ref{subsec3.3}.
			\end{remark}

			\begin{corollary} \label{quotients of the cohomology ring by ideals}
				For $0 \le k \le 2m-4$, let $\mathcal{C}_k$ be the basis of $\cohomgr{m}{2k}$ which is given as in \eqref{new basis of cohomology group of any degree}.
				Let $\mathcal{Q}_{1,0}$ and $\mathcal{Q}_{1,1}$ be the quotient $\Zbb$-modules
				\begin{eqnma} \label{definitions of Q_(1,0) and Q_(1,1)}
					\mathcal{Q}_{1,0} :=& \quotient{\cohomgr{m}{\bullet}}{\mathcal{M}_{1,0}};\\
					\mathcal{Q}_{1,1} :=& \quotient{\cohomgr{m}{\bullet}}{\mathcal{M}_{1,1}}
				\end{eqnma}
				where $\mathcal{M}_{1,0}$ is the $\Zbb$-submodule of $\cohomgr{m}{\bullet}$ which is generated by the basis
				\[
					\left( \bigsqcup_{k=0}^{2m-4} \mathcal{C}_k \right) \setminus \left\{ \omgao^k \ | \ 0 \le k \le m-2 \right\},
				\]
				and $\mathcal{M}_{1,1}$ is the $\Zbb$-submodule of $\cohomgr{m}{\bullet}$ which is generated by the basis
				\[
					\left( \bigsqcup_{k=0}^{2m-4} \mathcal{C}_k \right) \setminus \left\{ \omgaa^k \ \Big| \ 0 \le k \le \left\lfloor \frac{m-2}{2} \right\rfloor \right\}.
				\]
				Then we have
				\[
					\mathcal{Q}_{1,0} \simeq \quotientringao; \qquad
					\mathcal{Q}_{1,1} \simeq \quotientringaa
				\]
				as both a $\Zbb$-module and a ring \textnormal{(}Here, $(x^k)$ is an ideal in $\Zbb [x]$ generated by $x^k$\textnormal{).}
			\end{corollary}

			\begin{proof}
				Since $\mathcal{M}_{1,0}$ is an ideal in $\cohomgr{m}{\bullet}$, the quotient $\Zbb$-module $\mathcal{Q}_{1,0}$ has a quotient ring structure.
				By Proposition \hyperlink{new basis and relation between two bases : target}{\ref{new basis and relation between two bases}} (a), $\bigsqcup_{k=0}^{2m-4} \mathcal{C}_k$ is a $\Zbb$-module basis of $\cohomgr{m}{\bullet}$, and $\{ \omgao^k \ | \ 0 \le k \le m-2 \} \subset \bigsqcup_{k=0}^{m-2} \mathcal{C}_k$.
				Hence, $\mathcal{Q}_{1,0}$ is isomorphic to
				\begin{equation} \label{quotients of the cohomology ring by ideals : equation}
					\sspan \left( \left\{ \omgao^k \ | \ 0 \le k \le m-2 \right\} \right) \simeq \quotientringao
				\end{equation}
				as a $\Zbb$-module. Furthermore, $\mathcal{Q}_{1,0}$ is isomorphic to the right hand side of \eqref{quotients of the cohomology ring by ideals : equation} as a ring. The proof for $\mathcal{Q}_{1,1}$ is similar.
			\end{proof}

		\subsection{Vector bundles on Grassmannians} \label{subsec2.2}

			Let $E(d,m)$ be the \emph{universal bundle on $Gr(d,m)$} whose total space is
			\[
				\{ (x,v) \in Gr(d,m) \times \Cbb^m \ | \ v \in L_x \}.
			\]
			Denote the universal bundle on $Gr(d,V)$ by $E(d,V)$. We have the following canonical short exact sequence:
			\begin{equation} \label{short exact sequence 1}
				0 \ \to \ E(d,m) \ \to \ \bigoplus^m \mathcal{O}_{Gr(d,m)} \ \to \ Q(d,m) \ \to \ 0
			\end{equation}
			where $Q(d,m) := \quotient{(\bigoplus^m \mathcal{O}_{Gr(d,m)})}{E(d,m)}$, which is called the \emph{universal quotient bundle on $Gr(d,m)$}.

			Recall that automorphisms of $Gr(d,m)$ are classified in Theorem \ref{automorphism groups of Grassmannians}. The following Lemma is about the relations between $E(d,m), \, Q(d,m)$ and these automorphisms.

			\begin{lemma} \label{relations between E(d,m), Q(d,m) and automorphisms}
				Let $\varphi$ be an automorphism of $Gr(d,m)$.
				\begin{enumerate}
					\item[(a)] If $\varphi \in \mathbf{P} GL(m, \Cbb)$, then $\varphi^* (E(d,m)) \simeq E(d,m)$.
					\item[(b)] If $m = 2d$ and $\varphi \in \mathbf{P} GL(2d, \Cbb) \circ \phi$, then $\varphi^* (\ckE (d,2d)) \simeq Q(d,2d)$ where $\ckE (d,2d)$ is the dual bundle of $E(d,2d)$.
				\end{enumerate}
			\end{lemma}

			The proof of Lemma \ref{relations between E(d,m), Q(d,m) and automorphisms} is clear by definitions of $E(d,m)$ and $Q(d,m)$.

			\begin{proposition} [{\cite[Lemma 1.3 and 1.4]{TA}}] \label{total Chern classes of E(d,m) and Q(d,m)}
				In $Gr(d,m)$, the total Chern classes of $E(d,m)$ and $Q(d,m)$ are as follows:
				\begin{enumerate}
					\item[(a)] $c(E(d,m)) = 1 + \sum_{k=1}^d \, (-1)^k \, \omega \underbrace{_{1, \cdots , 1}}_k {_{, 0, \cdots , 0}}$.
					\item[(b)] $c(Q(d,m)) = 1 + \sum_{k=1}^{m-d} \, \omega_{k, 0, \cdots , 0}$.
				\end{enumerate}
			\end{proposition}

			Next, we consider vector bundles on Grassmannians of rank $2$. In \cite{BV1} and \cite{BV2}, W. Barth and A. Van de Ven found criteria of the decomposability of such vector bundles.

			Let $\mathcal{E}$ be a vector bundle on $\Pbb^k$ of rank $2$. For a projective line $\ell$ in $\Pbb^k$, the restriction $\mathcal{E} \big|_{\ell}$ is decomposable by Grothendieck theorem (\cite[Theorem 2.1.1]{OSS}), that is,
			\[
				\mathcal{E} \big|_{\ell} = \mathcal{O}_{\ell} (a_1) \oplus \mathcal{O}_{\ell} (a_2)
			\]
			for some integers $a_1$ and $a_2$ unique up to permutations. For such $a_1$ and $a_2$, define $b(\mathcal{E} \big|_{\ell})$ by the integer $\left\lfloor \frac{| a_1-a_2 |}{2} \right\rfloor$ and using this, let
			\begin{equation} \label{definition of B(E)}
				B(\mathcal{E}) := \max \left\{ b(\mathcal{E} \big|_{\ell}) \ | \ \Pbb^1 \simeq \ell \subset \Pbb^k \right\}.
			\end{equation}
			The following proposition tells us a sufficient condition for the decomposability of vector bundles on $\Pbb^k$ of rank $2$.

			\begin{proposition} [{\cite[Theorem 5.1]{BV1}}] \label{result on holomorphic vector bundles on P^k of rank 2}
				Let $\mathcal{E}$ be a vector bundle on $\Pbb^k$ of rank $2$ satisfying $B(\mathcal{E}) < \frac{k-2}{4}$. Then $\mathcal{E}$ is decomposable.
			\end{proposition}

			Also, there is a sufficient condition for which vector bundles on $Gr(d,m)$ of rank $2$ is either decomposable or isomorphic to some special form.

			\begin{proposition} [{\cite[Theorem 4.1]{BV1}}] \label{result on holomorphic vector bundles on Gr(d,m) of rank 2}
				Let $\mathcal{E}$ be a vector bundle on $Gr(d,m)$ of rank $2$ with $m-d \ge 2$ satisfying that the restrictions $\mathcal{E} \big|_{Y}$ are decomposable for all Schubert varieties $Y \subset Gr(d,m)$ of type $(m-d, \cdots , m-d, 0)$. Then either $\mathcal{E}$ is decomposable, or $d=2$ and $\mathcal{E} \simeq E(2,m) \otimes L$ for some line bundle $L$ on $Gr(2,m)$.
			\end{proposition}

			By Example \ref{examples of Schubert varieties of Gr(d,m)} (a), every Schubert variety of type $(m-d, \cdots , m-d, 0)$ is biholomorphic to a maximal projective space $\Pbb^{m-d}$, thus the conclusion of Proposition \ref{result on holomorphic vector bundles on P^k of rank 2} relates to the assumption of Proposition \ref{result on holomorphic vector bundles on Gr(d,m) of rank 2}.
			Proposition \ref{result on holomorphic vector bundles on P^k of rank 2} and \ref{result on holomorphic vector bundles on Gr(d,m) of rank 2} play important roles in proving Theorem \ref{main theorem : general case}.

	\vskip 0.5cm

	\section{Embeddings of {\texorpdfstring{\protect $Gr(2,m)$}{Gr(2,m)}} into {\texorpdfstring{\protect $Gr(2,n)$}{Gr(2,n)}}} \label{sec3}
		
		For $n \ge m > d$, let $\varphi \colon Gr(d,m) \hookrightarrow Gr(d,n)$ be an embedding.
		We characterize such an embedding $\varphi$ by means of the pullback bundle $E$ of the dual bundle of the universal bundle on $Gr(d,n)$ under the embedding $\varphi$.
		In particular, when $d=2$ and $m \ge 4$, the total Chern class of $E$ can be represented by integral coefficients, and the linearity of $\varphi$ is determined completely by these integers.

		\subsection{Linear embeddings} \label{subsec3.1}

			In {Section} \ref{sec1}, we defined a \emph{linear} embedding $\tilde{f}_W \colon Gr(d_1,m) \hookrightarrow Gr(d_2,n)$, which is induced by an injective linear map $f \colon \Cbb^m \hookrightarrow \Cbb^n$ and a $(d_2-d_1)$-dimensional subspace $W$ of $\Cbb^n$ satisfying $f(\Cbb^m) \cap W = 0$.
			When $d_1 = d_2 \, (=:d)$, we do not have to consider an extra summand $W$ of $L_{\tilde{f}_W (x)} \ (x \in Gr(d,m))$, so the definition of the linearity is simpler.
			Furthermore, when either $m=2d$ or $n=2d$, there is a non-linear, but still natural because of the existence of a dual map $\phi$, which is a non-linear automorphism of $\Aut (Gr(d,2d))$.

			\begin{definition} \label{definition of linearity of embeddings}
				Let $\varphi \colon Gr(d,m) \hookrightarrow Gr(d,n)$ be an embedding.
				\begin{enumerate}
					\item[(a)] An embedding $\varphi$ is \emph{linear} if $\varphi$ is induced by an injective linear map $f \colon \Cbb^m \hookrightarrow \Cbb^n$, that is,
					\[
						L_{\varphi(x)} = \text{the } d \text{-dimensional subspace } f(L_x) \text{ of } \Cbb^n
					\]
					for all $x \in Gr(d,m)$.
					\item[(b)] When $m=2d$ (resp. $n=2d$), an embedding $\varphi$ is \emph{twisted linear} if $\varphi$ is of the form $\varphi_0 \circ \phi$ (resp. $\phi \circ \varphi_0$) where $\varphi_0 \colon Gr(d,m) \hookrightarrow Gr(d,n)$ is a linear embedding and $\phi \colon Gr(d,2d) \to Gr(d,2d)$ is a dual map.
				\end{enumerate}
			\end{definition}

			\vskip 0.1cm

			\begin{remark} \label{twisted linearity when m=2d=n}
				Assume that $m = 2d = n$ and $\varphi \colon Gr(d,2d) \hookrightarrow Gr(d,2d)$ is a twisted linear embedding of the form $\phi \circ \varphi_0$ where $\varphi_0 \colon Gr(d,2d) \hookrightarrow Gr(d,2d)$ is a linear embedding.
				In this case, $\phi$ and $\varphi_0$ are automorphisms of $Gr(d,2d)$, and since $\varphi_0$ is linear, $\varphi_0 \in \mathbf{P} GL(2d, \Cbb)$.
				By Theorem \ref{automorphism groups of Grassmannians}, $\mathbf{P} GL(2d, \Cbb) \circ \phi = \phi \circ \mathbf{P} GL(2d, \Cbb)$, thus we have
				\[
					\varphi = \phi \circ \varphi_0 = \varphi_1 \circ \phi
				\]
				for some $\varphi_1 \in \mathbf{P} GL(2d, \Cbb)$.

				For this reason, we may assume that every twisted linear embedding $\varphi \colon Gr(d,2d) \hookrightarrow Gr(d,n)$ is a composition $\varphi_2 \circ \phi$ of maps for some linear embedding $\varphi_2 \colon Gr(d,2d) \hookrightarrow Gr(d,n)$.
			\end{remark}

			\vskip 0.1cm

			Of course, every embedding of $Gr(d,m)$ into $Gr(d,n)$ is not always linear.
			When $d=1$, S. Feder showed the existence of a non-linear embedding $\Pbb^m \hookrightarrow \Pbb^n$ for $n > 2m$ by Theorem \ref{S. Feder's result} (b).
			The following example provides some non-linear embeddings for $d \ge 2$.

			\begin{example} [Non-linear embeddings] \label{examples of non-linear embeddings}
				Let $\varphi \colon Gr(d,m) \hookrightarrow Gr(d,n)$ be an embedding.
				\begin{enumerate}
					\item[(a)] If either $m=2d$ or $n=2d$, then every twisted linear embedding is not linear.

					\item[(b)] Consider the Pl\"{u}cker embedding $i \colon Gr(d,m) \hookrightarrow \Pbb (\wedge^d \Cbb^m) = \Pbb^N$ where $N := \binom{m}{d} - 1$.
					There is a maximal projective space $Y \simeq \Pbb^{n-d}$ of $Gr(d,n)$, which is a Schubert variety of type $(n-d, \cdots , n-d, 0)$ (Example \ref{examples of Schubert varieties of Gr(d,m)} (a)), and let $j \colon Y \hookrightarrow Gr(d,n)$ be an inclusion.
					Apply S. Feder's result to this situation. If $n-d > 2N$, then there is a non-linear embedding $\psi \colon \Pbb^N \hookrightarrow \Pbb^{n-d} \simeq Y$ by Theorem \ref{S. Feder's result} (b).
					The composition $\varphi := j \circ \psi \circ i$ of maps is an embedding of $Gr(d,m)$ into $Gr(d,n)$, but it is not linear.

					\item[(c)] Consider an embedding $\varphi \colon Gr(2,3) \to Gr(2,4)$.
					Let $\psi := \varphi \circ \phi$ be the composition of maps where $\phi \colon \Pbb^2 \to Gr(2,3)$ is a dual map, then $\psi$ is an embedding of $\Pbb^2$ into $Gr(2,4)$.
					By Theorem \ref{H. Tango's result}, there are the following $4$ types of $X := \psi (\Pbb^2) = \varphi (Gr(2,3))$:
					\[
						X_{3,1}^0; \quad X_{3,1}^1; \quad \ckX_{3,1}^0; \quad \ckX_{3,1}^1
					\]
					which are given as in \eqref{subvarieties 1 of Gr(2,n) which is biholomorphic to P^(n-2)} and \eqref{subvarieties 2 of Gr(2,n) which is biholomorphic to P^(n-2)}, up to linear automorphisms of $Gr(2,4)$.
					In particular, $X_{3,1}^0$ is a Schubert variety of type $(2,0)$ and $\ckX_{3,1}^0$ is a Schubert variety of type $(1,1)$.
					If $X = \ckX_{3,1}^0$ up to linear automorphisms, then $\varphi$ is linear, and if $X = X_{3,1}^0$ up to linear automorphisms, then $\varphi$ is twisted linear.
					On the other hand, if $X = X_{3,1}^1$ or $\ckX_{3,1}^1$ up to linear automorphisms, then $\varphi$ is neither linear nor twisted linear.
					% In particular, when $X = X_{3,1}^1$, the composition of the corresponding embedding $\varphi$ and the Pl\"{u}cker embedding $Gr(2,4) \hookrightarrow \Pbb (\wedge^2 \Cbb^4)$ is given by
					% %
					% \begin{align*}
					% 	[x_0 : x_1 : x_2]
					% 	&\mapsto
					% 	\left[
					% 	\begin{pmatrix}
					% 		x_0 & x_1 & x_2 & 0\\
					% 		0 & x_0 & x_1 & x_2
					% 	\end{pmatrix}
					% 	\right]\\
					% 	&\mapsto [x_0^2 : x_0 x_1 : x_0 x_2 : x_1^2-x_0 x_2 : x_1 x_2 : x_2^2].
					% \end{align*}
					% %
					% Hence, we can show that $\varphi$ is not linear directly.
				\end{enumerate}
			\end{example}

			There is a relation between the (twisted) linearity of an embedding $\varphi$ of $Gr(d,m)$ into $Gr(d,n)$ and the image of $\varphi$.

			\begin{proposition} \label{relation between the linearity and the image of an embedding for general case}
				For $n \ge m \ge d$, let $\varphi \colon Gr(d,m) \hookrightarrow Gr(d,n)$ be an embedding.
				Then the image of $\varphi$ is equal to $Gr(d,H_{\varphi})$ for some $m$-dimensional subspace $H_{\varphi}$ of $\Cbb^n$ if and only if one of the following conditions holds:
				\begin{enumerate}
					\item[$\bullet$] $\varphi$ is linear;
					\item[$\bullet$] $m = 2d$ and $\varphi$ is twisted linear.
				\end{enumerate}
			\end{proposition}

			\begin{proof}
				During this proof, we denote the image of $\varphi$ by $X$.

				If $\varphi$ is linear, then $\varphi$ is induced by an injective linear map $f \colon \Cbb^m \hookrightarrow \Cbb^n$, thus $X = Gr(d, f(\Cbb^m))$.
				Moreover, if $\varphi$ is twisted linear with $m = 2d$, then $\varphi = \varphi_0 \circ \phi$ for some linear embedding $\varphi_0 \colon Gr(d,2d) \hookrightarrow Gr(d,n)$ (For the case when $n=m$, see Remark \ref{twisted linearity when m=2d=n}).
				So we have
				\[
					X = \varphi_0 (\phi (Gr(d,2d))) = \varphi_0 (Gr(d,2d)),
				\]
				thus the image of $\varphi = \varphi_0 \circ \phi$ is $Gr(d,H)$ for some $(2d)$-dimensional subspace $H$ of $\Cbb^n$.

				Conversely, assume that $X = Gr(d,H_{\varphi})$ where $H_{\varphi}$ is an $m$-dimensional subspace of $\Cbb^m$.
				Fix a biholomorphism $\psi_{\varphi} \colon Gr(d,H_{\varphi}) \to Gr(d,m)$ which is induced by a linear isomorphism $H_{\varphi} \to \Cbb^m$.
				Let $\varphi_1 \colon Gr(d,m) \to Gr(d, H_{\varphi})$ be a biholomorphism which is obtained by restricting the codomain of $\varphi$ to its image $Gr(d,H_{\varphi})$.
				Then $\psi_{\varphi} \circ \varphi_1$ is an automorphism of $Gr(d,m)$ as follows:
				\begin{center}
					\begin{tikzpicture}
						\node at (0,0) {$Gr(d,m)$};
						\draw[->] (1,0)--(2,0);
						\node at (1.5,0.2) {$\simeq$};
						\node at (1.5,-0.3) {$\varphi_1$};
						\node at (3,0) {$Gr(d,H_{\varphi})$};
						\draw[->] (4,0)--(5,0);
						\node at (4.5,0.2) {$\subset$};
						%\node at (4.5,-0.3) {$\iota$};
						\node at (6,0) {$Gr(d,n)$};
						\draw[->] (3,0.5)--(3,1.5);
						\node at (2.7,1) [rotate=90] {$\simeq$};
						\node at (3.3,1) {$\psi_{\varphi}$};
						\node at (3,2) {$Gr(d,m)$};
						\draw[->] (0,0.5)--(2,1.7);
						\node at (1.5,0.8) {$\circlearrowleft$};
						\node at (0.6,1.5) {$\psi_{\varphi} \circ \varphi_1$};
						\node at (6.2,-0.1) {.};
					\end{tikzpicture}
				\end{center}
				If $m \neq 2d$, then $\psi_{\varphi} \circ \varphi_1 \in \PGL (m,\Cbb)$ by Theorem \ref{automorphism groups of Grassmannians}, thus $\varphi$ is linear.
				If $m = 2d$, then $\psi_{\varphi} \circ \varphi_1 \in \PGL (2d,\Cbb) \sqcup (\PGL (2d, \Cbb) \circ \phi)$ by Theorem \ref{automorphism groups of Grassmannians}, thus $\varphi$ is either linear or twisted linear.
			\end{proof}

			Proposition \ref{relation between the linearity and the image of an embedding for general case} covers all the cases when $\varphi \colon Gr(d,m) \hookrightarrow Gr(d,n)$ is either linear or twisted linear except when $n = 2d > m$ and $\varphi$ is twisted linear.
			The following corollary is an analogous result for this exceptional case.

			\begin{corollary} \label{relation between the linearity and the image of an embedding for exceptional case}
				For $2d > m$, let $\varphi \colon Gr(d,m) \hookrightarrow Gr(d,2d)$ be an embedding.
				Then the image of $\varphi$ is equal to $\{ x \in Gr(d,2d) \ | \ V_{\varphi} \subset L_x \}$ for some $(2d-m)$-dimensional subspace $V_{\varphi}$ of $\Cbb^{2d}$ if and only if $\varphi$ is twisted linear.
			\end{corollary}

			\begin{proof}
				Since $2d > m$, $\varphi$ is twisted linear if and only if $\varphi = \phi \circ \varphi_0$ for some linear embedding $\varphi_0 \colon Gr(d,m) \hookrightarrow Gr(d,2d)$ or, equivalently, the image of $\phi \circ \varphi$ is equal to $Gr(d, H)$ for some $m$-dimensional subspace $H$ of $\Cbb^{2d}$ by Proposition \ref{relation between the linearity and the image of an embedding for general case}.
				To complete the proof, it suffices to show that given a $(2d-m)$-dimensional subspace $V_{\varphi}$ of $\Cbb^{2d}$,
				\[
					\phi ( \{ x \in Gr(d,2d) \ | \ V_{\varphi} \subset L_x \} ) = Gr(d,H)
				\]
				for some $m$-dimensional subspace $H$ of $\Cbb^{2d}$.
				Using the definition \eqref{definition of dual map} of a dual map $\phi$, we have
				\begin{align*}
					\phi \left( \{ x \in Gr(d,2d) \ | \ V_{\varphi} \subset L_x \} \right)
					= & \ \{ \phi (x) \in Gr(d,2d) \ | \ V_{\varphi} \subset L_x \}\\
					= & \ \{ \phi (x) \in Gr(d,2d) \ | \ \iota (L_x^{\perp}) \subset \iota (V_{\varphi}^{\perp}) \}\\
					= & \ \{ y \in Gr(d,2d) \ | \ L_y \subset \iota (V_{\varphi}^{\perp}) \}\\
					& \hskip 0.5cm (\because L_{\phi (x)} = \iota (L_x^{\perp}) \text{ by Equation } \eqref{definition of dual map})\\
					= & \ Gr(d, \iota (V_{\varphi}^{\perp})),
				\end{align*}
				where $\iota (V_{\varphi}^{\perp})$ is a subspace of $\Cbb^{2d}$ of dimension $2d-(2d-m) = m$ as desired.
			\end{proof}

		\subsection{Equivalent conditions for the linearity} \label{subsec3.2}

			Let $\varphi \colon Gr(2,m) \hookrightarrow Gr(2,n)$ be an embedding and $E := \varphi^* (\ckE (2,n))$ where $\ckE (2,n)$ is the dual bundle of $E(2,n)$.
			To distinguish Schubert cycles on two Grassmannians $Gr(2,m)$ and $Gr(2,n)$, denote Schubert cycles on $Gr(2,n)$ (resp. $Gr(2,m)$) by $\tilde{\omega}_{i,j}$ (resp. $\omega_{k,l}$) (Of course, all properties in {Section} \ref{sec2} hold for $Gr(2,n)$ and $\tilde{\omega}_{i,j}$).
			By Proposition \ref{total Chern classes of E(d,m) and Q(d,m)} (a),
			\[
				c(\ckE (2,n)) = 1 + \tlomgao + \tlomgaa
			\]
			and
			\begin{eqnma} \label{Chern classes of E}
				c_1(E) &= \varphi^* (\tlomgao) = a \, \omgao,\\
				c_2(E) &= \varphi^* (\tlomgaa) = b \, \omgao^2 + c \, \omgaa = b \, \omgbo + (b+c) \, \omgaa
			\end{eqnma}
			for some $a,b,c \in \Zbb$.
			By Proposition \hyperlink{new basis and relation between two bases : target}{\ref{new basis and relation between two bases}} (a), $\{ \omega_{0,0}, \omgao, \omgao^2, \omgaa \}$ is a $\Zbb$-module basis of $\bigoplus_{k=0}^2 \cohomgr{m}{2k}$ if and only if $m \ge 4$.
			So in this case, the coefficients $a,b$ and $c$ given as in \eqref{Chern classes of E} are determined uniquely for each $\varphi$.
			For this reason, we always assume that $m \ge 4$.

			\begin{lemma} \label{cohomology class of X}
				For $n \ge m \ge 4$, let $\varphi \colon Gr(2,m) \hookrightarrow Gr(2,n)$ be an embedding. Then the Poincar\'{e} dual to the homology class of $X := \varphi (Gr(2,m))$ is
				\[
					\sum_{i=0}^{n-m} (X \cdot \tilde{\omega}_{n-2-i,2m-n-2+i}) \, \tilde{\omega}_{2n-2m-i,i}
				\]
				where $X \cdot \tilde{\omega}_{n-2-i,2m-n-2+i}$ is the intersection number in $Gr(2,n)$.
			\end{lemma}

			\begin{proof}
				Since the codimension of $X \simeq Gr(2,m)$ in $Gr(2,n)$ is $2n-2m$, the Poincar\'{e} dual to the homology class of $X$ is
				\[
					\sum_{i=0}^{n-m} d_i \, \tilde{\omega}_{2n-2m-i,i}
				\]
				for some integers $d_i$. By Corollary \ref{dual Schubert cycle of a Schubert cycle},
				\[
					X \cdot \tilde{\omega}_{n-2-j,2m-n-2+j} = \sum_{i=0}^{n-m} d_i \, (\tilde{\omega}_{2n-2m-i,i} \cdot \tilde{\omega}_{n-2-j,2m-n-2+j}) = d_j
				\]
				as desired.
			\end{proof}

			Now we ready to prove the following proposition on the equivalent conditions for the (twisted) linearity of an embedding.

			\begin{proposition} \label{equivalent conditions for the linearity of an embedding}
				For $n \ge m \ge 4$ \textnormal{(}resp. $n \ge m = 4$\textnormal{)}, let $\varphi \colon Gr(2,m) \hookrightarrow Gr(2,n)$ be an embedding and $a,b,c$ be the the integers given as in \eqref{Chern classes of E}. The following are equivalent:
				\begin{enumerate}
					\item[(a)] $\varphi$ is linear \textnormal{(}resp. $\varphi$ is twisted linear\textnormal{)};
					\item[(b)] $E \simeq \ckE (2,m)$ \textnormal{(}resp. $E \simeq Q(2,4)$\textnormal{)};
					\item[(c)] $(a,b,c) = (1,0,1)$ \textnormal{(}resp. $(a,b,c) = (1,1,-1)$\textnormal{)}.
				\end{enumerate}
			\end{proposition}

			\begin{proof}
				During this proof, we denote the image of $\varphi$ by $X$.

				\begin{enumerate}
					\item[$\bullet$] (a)$\Implies$(b) :
					Assume that $\varphi$ is linear, then $X = Gr(2, H_{\varphi})$ for some $m$-dimensional subspace $H_{\varphi}$ of $\Cbb^n$ by Proposition \ref{relation between the linearity and the image of an embedding for general case}.
					The total space of $E(2,n) |_X$ is
					\begin{align*}
						& \{ (x,v) \in X \times \Cbb^n \ | \ v \in L_x \subset \Cbb^n \}\\
						=& \{ (x,v) \in Gr(2,H_{\varphi}) \times H_{\varphi} \ | \ v \in L_x \subset H_{\varphi} \}
					\end{align*}
					which is exactly equal to the total space of $E(2,H_{\varphi})$.
					So $E(2,n) |_X = E(2,H_{\varphi})$ and we have
					\[
						E = \varphi^* (\ckE (2,n)) = \varphi^* (\ckE (2,n) |_X) = \varphi^* (\ckE (2,H_{\varphi})) \simeq \ckE (2,m).
					\]
					If $m=4$ and $\varphi$ is twisted linear, then $\varphi \circ \phi$ is linear.
					By the previous result, $(\varphi \circ \phi)^* (E) \simeq \ckE (2,4)$.
					So $\varphi^* (E) \simeq \phi^* (\ckE (2,4))$, which is isomorphic to $Q(2,4)$ by Lemma \ref{relations between E(d,m), Q(d,m) and automorphisms} (b).

					\item[$\bullet$] (b)$\Implies$(c) :
					By Proposition \ref{total Chern classes of E(d,m) and Q(d,m)}, we have
					\begin{align*}
						c(\ckE (2,m)) &= 1 + \omgao + \omgaa\\
						c(Q(2,m)) &= 1 + \omgao + \omgbo = 1 + \omgao + \omgao^2 - \omgaa.
					\end{align*}
					So if $E \simeq \ckE (2,m)$ (resp. $E \simeq Q(2,4)$), then $(a,b,c) = (1,0,1)$ (resp. $(a,b,c) = (1,1,-1)$).

					\item[$\bullet$] (c)$\Implies$(a) :
					Assume that $(a,b,c) = (1,0,1)$.
					Then $\varphi^* (\tlomgao) = \omgao$ and $\varphi^* (\tlomgaa) = \omgaa$. Note that $\cohomgr{n}{\bullet}$ (resp. $\cohomgr{m}{\bullet}$) is generated by $\tlomgao$ and $\tlomgaa$ (resp. $\omgao$ and $\omgaa$) as a ring.
					Moreover, the multiplicative structures of $\cohomgr{n}{\bullet}$ and those of $\cohomgr{m}{\bullet}$ are exactly same when we adopt Convention \ref{convention for Schubert cycles}.
					So we have
					\begin{equation} \label{invariance of Schubert cycles under phi}
						\varphi^* (\tilde{\omega}_{i,j}) = \omega_{i,j}
					\end{equation}
					for all $n-2 \ge i \ge j \ge 0$.
					The Poincar\'{e} dual to the homology class of $X$ is
					\begin{equation} \label{cohomology class of X : equation}
						\sum_{i=0}^{n-m} (X \cdot \tilde{\omega}_{n-2-i,2m-n-2+i}) \, \tilde{\omega}_{2n-2m-i,i}
					\end{equation}
					by Lemma \ref{cohomology class of X}, and so we have
					\begin{align*}
						\eqref{cohomology class of X : equation}
						= & \ \sum_{i=0}^{n-m} (\varphi^* (\tilde{\omega}_{n-2-i,2m-n-2+i})) \, \tilde{\omega}_{2n-2m-i,i}\\
						= & \ \sum_{i=0}^{n-m} (\omega_{n-2-i,2m-n-2+i}) \, \tilde{\omega}_{2n-2m-i,i} \quad (\because \text{ Equation } \eqref{invariance of Schubert cycles under phi})\\
						= & \ (\omega_{m-2,m-2}) \, \tilde{\omega}_{n-m,n-m} \quad (\because \ \omega_{n-2-i,2m-n-2+i} = 0\\
						& \hskip 2.5cm \text{ unless } m-2 \ge n-2-i \ge 2m-n-2+i \ge 0)\\
						= & \ \tilde{\omega}_{n-m,n-m}.
					\end{align*}
					By Proposition \ref{rigidity of sub-Grassmannians}, any subvariety of $Gr(2,n)$ which corresponds to $\tilde{\omega}_{n-m,n-m}$ is of the form $Gr(2,H)$ where $H$ is an $m$-dimensional subspace of $\Cbb^n$.
					Applying Proposition \ref{relation between the linearity and the image of an embedding for general case} and the implication of (a)$\Implies$(c), we obtain the desired result.

					Assume that $m=4$ and $(a,b,c) = (1,1,-1)$.
					To prove the implication of (c)$\Implies$(a) for this case, it suffices to show that $\varphi \circ \phi$ is linear.
					The total Chern class of $(\varphi \circ \phi)^* (E) = \phi^* (\varphi^* (E))$ equals
					\begin{align*}
						c((\varphi \circ \phi)^* (E))
						= & \ c(\phi^* (\varphi^* (E))) = \phi^* (c(\varphi^* (E)))\\
						= & \ \phi^* (1 + \omgao + \omgao^2 - \omgaa)\\
						= & \ 1 + \omgao + \omgao^2 - (\omgao^2 - \omgaa)\\
						%(\because \ \phi^* (\omgao) = \omgao, \quad \phi^* (\omgaa) = \omgao^2 - \omgaa)\\
						= & \ 1 + \omgao + \omgaa,
					\end{align*}
					thus $\varphi \circ \phi$ is linear by the implication of (c)$\Implies$(a) for the linear case.
					Hence, $\varphi$ is twisted linear.
				\end{enumerate}
			\end{proof}

			\vskip 0.1cm

			\begin{remark} \label{application of N. Mok's result}
				Assume that $n < 2m-2$ and $a=1$.
				Using Theorem \ref{N. Mok's result}, we can prove the implication of (c)$\Implies$(a) in Proposition \ref{equivalent conditions for the linearity of an embedding} for this case without computing the Poincar\'{e} dual to the homology class of the image of $\varphi$.
				Since $a = 1$,  $\varphi$ maps each projective line in $Gr(2,m)$ to a projective line in $Gr(2,n)$, thus for each $x \in Gr(2,m)$, the differential $d \varphi$ preserves the decomposability of tangent vectors of $Gr(2,m)$ at $x$.
				So by Theorem \ref{N. Mok's result}, either $\varphi$ is linear up to automorphisms of $Gr(2,m)$ or $Gr(2,n)$, or the image of $\varphi$ lies on some projective space in $Gr(2,n)$.
				But the latter case is impossible because the dimension of $\varphi (Gr(2,m))$ is greater than the dimension of each maximal projective space in $Gr(2,n)$ (or equivalently, $2m-4 > n-2$).
				By Theorem \ref{automorphism groups of Grassmannians}, there is a non-linear automorphism of $Gr(2,m)$ only for the case when $m = 4$.
				Hence, if $m \ge 5$, then $\varphi$ is linear, and if $m = 4$, then $\varphi$ is either linear or twisted linear.
			\end{remark}

		\subsection{Equations in {\texorpdfstring{$a,b$}{a,b}} and {\texorpdfstring{$c$}{c}}} \label{subsec3.3}

			Let $a,b$ and $c$ be the integers which are given as in \eqref{Chern classes of E}, and $N$ be the pullback bundle of the normal bundle of $X := \varphi (Gr(2,m))$ in $Gr(2,n)$ under the embedding $\varphi$.
			% comment
			% In this {subsection}, we construct an equation of the Euler class $e(N_{\Rbb})$ (Proposition \ref{Euler class of N_R}) and an equation of the total Chern class $c(N)$ (Lemma \ref{total Chern class of N}), and under the assumption $n \le \frac{3m-2}{2}$, we deduce refined equations in one variable (Proposition \ref{refined total Chern class of N}) from the equation of $c(N)$.
			% Using the definitions of the Chern classes $c(N)$ in terms of the Euler class $e(N_{\Rbb})$ (Note \ref{result on the Chern classes of high degrees}), we obtain several Diophantine equations in variables $a,b$ and $c$.

			There are two methods to construct the Chern classes of a complex vector bundle $\mathcal{E}$:
			one is via Chern-Weil theory and the other is via the Euler class $e(\mathcal{E}_{\Rbb})$ of $\mathcal{E}_{\Rbb}$.
			For more details on the first and the second methods to construct Chern classes, see \cite[Section 20]{BT} and \cite[Section 14]{MS}, respectively.
			For more details on Euler classes, see \cite[Section 11]{BT} or \cite[Section 9]{MS}.

			By the second method (or deriving from the first method), we have the following result on the $k$\textsuperscript{th} Chern classes $c_k (\mathcal{E})$ of $\mathcal{E}$ for $k \ge \rank (\mathcal{E})$:

			\vskip 0.1cm

			\begin{note} [{\cite[(20.10.4) and (20.10.6)]{BT} or \cite[page 158]{MS}}] \label{result on the Chern classes of high degrees}
				Let $Z$ be a real manifold and $\mathcal{E}$ be a complex vector bundle on $Z$.
				Then we have
				\[
					c_k (\mathcal{E}) =
					\begin{cases}
						e(\mathcal{E}_{\Rbb}), & \text{if } k = \rank (\mathcal{E})\\
						0, & \text{if } k > \rank (\mathcal{E})
					\end{cases}
				\]
				where $\mathcal{E}_{\Rbb}$ is the real vector bundle on $Z$ which corresponds to $\mathcal{E}$.
			\end{note}

			\vskip 0.1cm

			As in Note \ref{result on the Chern classes of high degrees}, let $N_{\Rbb}$ be the real vector bundle which corresponds to $N$.
			Then we can compute the Euler class of $N_{\Rbb}$ as follows:

			\begin{proposition} \label{Euler class of N_R}
				For $n \ge m \ge 4$, the Euler class of $N_{\Rbb}$ is
				\begin{equation} \label{Euler class of N_R : equation}
					e(N_{\Rbb}) = \sum_{i=0}^{n-m} (X \cdot \tilde{\omega}_{n-2-i,2m-n-2+i}) \, \varphi^* (\tilde{\omega}_{2n-2m-i,i}).
				\end{equation}
				In particular, when $n=m+1$,
				\begin{align*}
					e(N_{\Rbb})
					= & \ (X \cdot \tilde{\omega}_{m-1,m-3}) \, \varphi^* (\tlomgbo) + (X \cdot \tilde{\omega}_{m-2,m-2}) \, \varphi^* (\tlomgaa)\\
					% = & (X \cdot \tilde{\omega}_{m-1,m-3}) \, \{ (a^2-b) \, \omgao^2 - c \, \omgaa \} + (X \cdot \tilde{\omega}_{m-2,m-2}) \, (b \, \omgao^2 + c \, \omgaa)\\
					= & \ \{ (X \cdot \tilde{\omega}_{m-1,m-3}) \, (a^2-b) + (X \cdot \tilde{\omega}_{m-2,m-2}) \, b \} \, \omgao^2\\
					& \hskip 0.5cm + ( X \cdot \tilde{\omega}_{m-2,m-2} - X \cdot \tilde{\omega}_{m-1,m-3} ) \, c \, \omgaa.
				\end{align*}
			\end{proposition}

			\begin{proof}
				By \cite[Theorem 1.3]{FE},
				\[
					e(N_{\Rbb}) = \varphi^* (\varphi_* (1))
				\]
				where $1$ is the cohomology class in $\cohomgr{m}{0}$ which corresponds to $Gr(2,m)$ itself.
				Since $\varphi_* (1)$ is the cohomology class in $\cohomgr{n}{2(2n-2m)}$ which corresponds to $X$,
				\begin{align*}
					e(N_{\Rbb})
					&= \varphi^* \left( \sum_{i=0}^{n-m} (X \cdot \tilde{\omega}_{n-2-i,2m-n-2+i}) \, \tilde{\omega}_{2n-2m-i,i} \right)\\
					&= \sum_{i=0}^{n-m} (X \cdot \tilde{\omega}_{n-2-i,2m-n-2+i}) \, \varphi^* (\tilde{\omega}_{2n-2m-i,i})
				\end{align*}
				by Lemma \ref{cohomology class of X}.
			\end{proof}

			% Comment
			% The cohomology ring $\cohomgr{n}{\bullet}$ of $Gr(2,n)$ is generated by $\{ \tlomgao, \, \tlomgaa \}$ as a ring, thus any $\tilde{\Gamma} \in \cohomgr{n}{\bullet}$ can be written as $f_{\tilde{\Gamma}} (\tlomgao, \, \tlomgaa)$ for some polynomial $f_{\tilde{\Gamma}} \in \Zbb [x,y]$.
			% Since $\varphi^* \colon \cohomgr{n}{\bullet} \to \cohomgr{m}{\bullet}$ is a ring homomorphism, we have $\varphi^* (\tilde{\Gamma}) = f_{\tilde{\Gamma}} (a \, \omgao, \, b \, \omgao^2 + c \, \omgaa)$, which is expressed as $g_{\tilde{\Gamma}, \varphi} (\omgao, \, \omgaa)$ for some $g_{\tilde{\Gamma}, \varphi} \in \Zbb [x,y]$.

			% Apply this method to the Euler class $e(N_{\Rbb})$.
			In each summand of \eqref{Euler class of N_R : equation}, $X \cdot \tilde{\omega}_{n-2-i,2m-n-2+i} \in \Zbb$ and $\varphi^* (\tilde{\omega}_{2n-2m-i,i}) \in \cohomgr{m}{2(2n-2m)}$, so we can express $e(N_{\Rbb})$ as
			\begin{equation} \label{naive expression of Euler class}
				e(N_{\Rbb}) = \sum_{i=0}^{n-m} A_i \, \omgao^{2n-2m-2i} \, \omgaa^i
			\end{equation}
			for some polynomials $A_i \in \Zbb [a,b,c]$ (but the expression may not be unique).

			\begin{lemma} \label{total Chern class of N}
				Let $n \ge m \ge 4$.
				Then the total Chern class $c(N)$ of $N$ satisfies the following equation:
				\begin{eqnmg} \label{total Chern class of N : equation}
					c(N) (1 + (4b-a^2) \, \omgao^2 + 4c \, \omgaa) (1 + \omgao + \omgaa)^m\\
					= (1 + a \, \omgao + b \, \omgao^2 + c \, \omgaa)^n (1 - \omgao^2 + 4 \, \omgaa)
				\end{eqnmg}
				which is satisfied in $\cohomgr{m}{\bullet}$.
				Moreover, the first and the second Chern classes of $N$ are
				\begin{align*}
					c_1 (N) &= (an-m) \, \omgao;\\
					c_2 (N) &= \left\{ \binom{n}{2} a^2 - amn + m^2 - \binom{m}{2} + a^2 - 1 + b(n-4) \right\} \, \omgao^2\\
					& \hskip 2.0cm + \{ c(n-4) - m + 4 \} \, \omgaa.
				\end{align*}
			\end{lemma}

			\begin{proof}
				Taking the tensor product of \eqref{short exact sequence 1} with $\ckE (2,m)$, we obtain a short exact sequence
				\begin{equation} \label{short exact sequence 2}
					0 \ \to \ E(2,m) \otimes \ckE (2,m) \ \to \ \bigoplus^m \ckE (2,m) \ \to \ Q(2,m) \otimes \ckE (2,m) \ \to \ 0,
				\end{equation}
				and after replacing $m$ with $n$, we also obtain a short exact sequence
				\begin{equation} \label{short exact sequence 3}
					0 \ \to \ E(2,n) \otimes \ckE (2,n) \ \to \ \bigoplus^n \ckE (2,n) \ \to \ Q(2,n) \otimes \ckE (2,n) \ \to \ 0.
				\end{equation}
				Since $T_{Gr(2,n)} \simeq Q(2,n) \otimes \ckE (2,n)$ and $T_{Gr(2,m)} \simeq Q(2,m) \otimes \ckE (2,m)$,
				\begin{align*}
					c(\varphi^*(T_{Gr(2,n)})) &= \frac{\varphi^* (c(\ckE (2,n)))^n}{\varphi^* (c(E(2,n) \otimes \ckE(2,n)))} = \frac{c(E)^n}{c(\ckE \otimes E)};\\
					c(T_{Gr(2,m)}) &= \frac{c(\ckE (2,m))^m}{c(E(2,m) \otimes \ckE (2,m))}
				\end{align*}
				by \eqref{short exact sequence 2} and \eqref{short exact sequence 3}.
				So we have the following equation:
				\[
					c(N) = \frac{c(\varphi^* (T_{Gr(2,n)}))}{c(T_{Gr(2,m)})} = \frac{c(E)^n / c(\ckE \otimes E)}{c(\ckE (2,m))^m / c(E(2,m) \otimes \ckE (2,m))},
				\]
				that is,
				\begin{equation} \label{equation of c(N)}
					c(N) c(\ckE \otimes E) c(\ckE (2,m))^m = c(E)^n c(E(2,m) \otimes \ckE (2,m)).
				\end{equation}
				Note that
				\begin{eqnma} \label{equations of each term in c(N)}
					c(E) &= 1 + a \, \omgao + b \, \omgao^2 + c \, \omgaa;\\
					c(\ckE \otimes E) &= 1 + (4b-a^2) \, \omgao^2 + 4c \, \omgaa;\\
					c(\ckE (2,m)) &= 1 + \omgao + \omgaa;\\
					c(E(2,m) \otimes \ckE (2,m)) &= 1 - \omgao^2 + 4 \, \omgaa.
				\end{eqnma}
				Putting \eqref{equations of each term in c(N)} into \eqref{equation of c(N)}, we obtain Equation \eqref{total Chern class of N : equation}.

				Comparing the cohomology classes of degree $2$ in both sides of \eqref{total Chern class of N : equation}, we have
				\[
					c_1 (N) + m \, \omgao = an \, \omgao,
				\]
				so we obtain
				\begin{equation} \label{first Chern class of N}
					c_1 (N) = (an-m) \, \omgao.
				\end{equation}
				Comparing the cohomology classes of degree $4$ in both sides of \eqref{total Chern class of N : equation}, we have
				\begin{eqnmg} \label{not simplified second Chern class of N}
					\displaystyle c_2 (N) + c_1 (N) \cdot m \, \omgao + (4b-a^2) \, \omgao^2 + 4c \, \omgaa + \binom{m}{2} \, \omgao^2 + m \, \omgaa\\
					\displaystyle = \binom{n}{2} a^2 \, \omgao^2 + bn \, \omgao^2 + cn \, \omgaa - \omgao^2 + 4 \, \omgaa.
				\end{eqnmg}
				Putting \eqref{first Chern class of N} into \eqref{not simplified second Chern class of N}, we obtain
				\begin{align*}
					c_2 (N)
					&= \left\{ \binom{n}{2} a^2 - amn + m^2 - \binom{m}{2} + a^2 - 1 + b(n-4) \right\} \, \omgao^2\\
					& \hskip 2.0cm + \{ c(n-4) - m + 4 \} \, \omgaa
				\end{align*}
				as desired.
			\end{proof}

			By replacing $c(N)$ by $\Gamma$, regard \eqref{total Chern class of N : equation} as an equation with a variable $\Gamma$.
			Write a solution $\Gamma$ of \eqref{total Chern class of N : equation} as $\Gamma = \sum_{k=0}^{2m-4} \Gamma_k$ where $\Gamma_k \in \cohomgr{m}{2k}$ for all $0 \le k \le 2m-4$.
			In the proof of Lemma \ref{total Chern class of N}, we compute $\Gamma_k = c_k (N)$ for $k=1$ and $2$ by the following steps:

			\begin{enumerate}
				\item[\fbox{1}] Obtain an equation $(\star_k)$, which is satisfied in $\cohomgr{m}{2k}$, after comparing the cohomology classes of degree $2k$ in both sides of \eqref{total Chern class of N : equation};
				\item[\fbox{2}] compute $\Gamma_k$ by putting $\Gamma_i$ into $(\star_k)$ for all $0 \le i < k$.
			\end{enumerate}

			Repeat these operations from $k=3$ to $2m-4$.
			After that, we can express $\Gamma_k \, (= c_k (N))$ for $0 \le k \le 2m-4$ as follows:
			\begin{eqnma} \label{naive expression of Chern classes}
				\Gamma_0 = & \ 1 = B_{0,0} \, \omgao^0 = B_{0,0} \, \omgaa^0\\
				\Gamma_1 = & \ B_{1,0} \, \omgao\\
				\Gamma_2 = & \ B_{2,0} \, \omgao^2 + B_{2,1} \, \omgaa\\
				\vdots &\\
				\Gamma_k = & \ B_{k,0} \, \omgao^k + B_{k,1} \, \omgao^{k-2} \, \omgaa + \cdots + B_{k,h_k} \, \omgao^{k-2 h_k} \, \omgaa^{h_k}\\
				\vdots &\\
				\Gamma_{2m-4} = & \, B_{2m-4,0} \, \omgao^{2m-4} + B_{2m-4,1} \, \omgao^{2m-6} \omgaa + \cdots + B_{2m-4,m-2} \, \omgaa^{m-2}
			\end{eqnma}
			for some $B_{k,i} \in \Zbb [a,b,c]$ for all $0 \le k \le 2m-4$, $0 \le i \le h_k := \left\lfloor \frac{k}{2} \right\rfloor$.
			In addition, $\Gamma_k = 0$ for all $k > 2m-4$.

			Since the rank of $N$ is $2n-2m$, we have by Note \ref{result on the Chern classes of high degrees},
			\begin{equation} \label{Chern classes of N of high degrees}
				c_k (N) =
				\begin{cases}
					e (N_{\Rbb}), & \text{if } k=2n-2m\\
					0, & \text{if } 2n-2m < k \le 2m-4
				\end{cases}
				\centermark{2}{.}
			\end{equation}
			If $2n-2m > 2m-4$, then we cannot obtain any further information from \eqref{Chern classes of N of high degrees}.
			On the other hand, if $2n-2m \le 2m-4$, then we obtain several Diophantine equations in $a,b$ and $c$ by putting \eqref{naive expression of Euler class} and \eqref{naive expression of Chern classes} into \eqref{Chern classes of N of high degrees}.
			% Comment
			% For this reason, the condition $2n-2m \le 2m-4$ (or equivalently, $n \le 2m-2$) is essential in order to find all possible pairs $(a,b,c)$ of integers from these Diophantine equations.

			If $m-2 < 2n-2m \le 2m-4$, then the choice of coefficients $A_i$ in \eqref{naive expression of Euler class} is not unique and the choice of coefficients $B_{k,i}$ in \eqref{naive expression of Chern classes} is not unique for all $2n-2m \le k \le 2m-4$ by Proposition \hyperlink{new basis and relation between two bases : target}{\ref{new basis and relation between two bases}} (a) (or Remark \ref{reason why 2n-2m <= m-2 is important}).
			% Comment
			% To put \eqref{naive expression of Euler class} and \eqref{naive expression of Chern classes} into \eqref{Chern classes of N of high degrees}, we need to write $e(N_{\Rbb})$ and $\Gamma_{2n-2m}$ with respect to the same basis of $\cohomgr{m}{2(2n-2m)}$, and write $\Gamma_k$ with respect to a basis of $\cohomgr{m}{2k}$ for each $2n-2m < k \le 2m-4$ again.
			%
			If $2n-2m \le m-2$, then we do not have to care the uniqueness of coefficients in $e(N_{\Rbb})$ and $\Gamma_k \ (2n-2m \le k \le m-2)$.
			Thus we have the following Diophantine equations directly:
			\begin{equation} \label{Diophantine equations when 2n-2m <= m-2}
				\begin{array}{r@{\,}lr@{\,}lcr@{\,}l}
					A_0 & = B_{2n-2m,0}; & A_1 & = B_{2n-2m,1}; & \cdots & A_{n-m} & = B_{2n-2m,n-m};\\
					0 & = B_{k,0}; & 0 & = B_{k,1}; & \cdots & 0 & = B_{h,h_k};
				\end{array}
			\end{equation}
			for all $2n-2m < k \le m-2$.
			However, if $n-m$ is big, then it is difficult to find all the Diophantine equations in \eqref{Diophantine equations when 2n-2m <= m-2}.
			So we need simpler equations than \eqref{total Chern class of N : equation}.

			Using Equation \eqref{total Chern class of N : equation} and the bases \eqref{new basis of cohomology group of low degree} of cohomology groups together with Corollary \ref{quotients of the cohomology ring by ideals}, we obtain two refined equations in one variable as follows:

			\begin{proposition} \label{refined total Chern class of N}
				Let $m \le n \le \frac{3m-2}{2}$. For $0 \le k \le 2n-2m$, write the $k$\textsuperscript{th} Chern class $c_k (N)$ of $N$ as
				\begin{equation} \label{definitions of alpha_k and beta_k}
					c_k (N) =
					\begin{cases}
						\alpha_0 = 1 = \beta_0, & \text{if } k = 0\\
						\alpha_1 \, \omgao = \beta_1 \, \omgao, & \text{if } k = 1\\
						\alpha_k \, \omgao^k + \cdots + \beta_k \, \omgaa^{k/2}, & \text{if } k \ge 2 \text{ is even}\\
						\alpha_k \, \omgao^k + \cdots + \beta_k \, \omgao \, \omgaa^{(k-1)/2}, & \text{if } k \ge 3 \text{ is odd}
					\end{cases}
				\end{equation}
				with respect to the basis \eqref{new basis of cohomology group of low degree}. Then we have two equations
				\begin{eqnmg} \label{refined total Chern class of N : equation 1}
					\left( \sum_{k=0}^{2n-2m} \alpha_k \, \omgao^k \right) (1 + (4b-a^2) \, \omgao^2) (1 + \omgao)^{m-1}\\
					= (1 + a \, \omgao + b \, \omgao^2)^n (1 - \omgao)
				\end{eqnmg}
				which is satisfied in $\quotientringao$, and
				\begin{equation} \label{refined total Chern class of N : equation 2}
					\left( \sum_{k=0}^{n-m} \beta_{2k} \, \omgaa^k \right) (1 + 4c \, \omgaa) (1 + \omgaa)^m =  (1 + c \, \omgaa)^n (1 + 4 \, \omgaa)
				\end{equation}
				which is satisfied in $\quotientringaa$.
			\end{proposition}

			\begin{proof}
				By Proposition \hyperlink{new basis and relation between two bases : target}{\ref{new basis and relation between two bases}} (a), we can express each $c_k (N)$ where $0 \le k \le 2n-2m$ uniquely as in \eqref{definitions of alpha_k and beta_k} because $2n-2m \le m-2$.

				By Corollary \ref{quotients of the cohomology ring by ideals}, there are ring isomorphisms $\rho_0 \colon \mathcal{Q}_{1,0} \to \quotientringao$ and $\rho_1 \colon \mathcal{Q}_{1,1} \to \quotientringaa$ where $\mathcal{Q}_{1,0}$ and $\mathcal{Q}_{1,1}$ are the quotient rings given as in \eqref{definitions of Q_(1,0) and Q_(1,1)}. Consider the images of both sides of \eqref{total Chern class of N : equation} under the composition of the isomorphism $\rho_0$ with the canonical projection $\cohomgr{m}{\bullet} \to \mathcal{Q}_{1,0}$. To find their images, it suffices to consider the terms involving in $\omgao$, so we have
				\[
					\left( \sum_{k=0}^{2n-2m} \alpha_k \, \omgao^k \right) (1 + (4b-a^2) \, \omgao^2) (1 + \omgao)^m = (1 + a \, \omgao + b \, \omgao^2)^n (1 - \omgao^2),
				\]
				and after dividing both sides by $1 + \omgao$, we obtain Equation \eqref{refined total Chern class of N : equation 1}. Similarly, we obtain \eqref{refined total Chern class of N : equation 2} by considering the image of both sides of \eqref{total Chern class of N : equation} under the composition of $\rho_1$ with the canonical projection $\cohomgr{m}{\bullet} \to \mathcal{Q}_{1,1}$.
			\end{proof}

			\vskip 0.1cm

			\begin{remark}
				Assume that $n \le \frac{3m-2}{2}$.
				When we compare the notations in \eqref{naive expression of Chern classes} and \ref{definitions of alpha_k and beta_k}, $\alpha_k = B_{k,0}$ and $\beta_k = B_{k,\lfloor k/2 \rfloor}$ for all $0 \le k \le 2n-2m$. In particular, $\beta_{2i} = B_{2i,i}$ for all $0 \le i \le n-m$.
			\end{remark}

			\vskip 0.1cm

			Each of \eqref{refined total Chern class of N : equation 1} and \eqref{refined total Chern class of N : equation 2} is satisfied in a quotient ring of the form $\quotient{\Zbb [x]}{(x^k)}$, which is isomorphic to the $\Zbb$-module which consists of all polynomials in $\Zbb [x]$ of degree $< k$ as a $\Zbb$-module.
			When we express each element in $\quotient{\Zbb [x]}{(x^k)}$ as a polynomial of degree $< k$, we use \eqref{refined total Chern class of N : equation 1} and \eqref{refined total Chern class of N : equation 2} as follows:
			\begin{enumerate}
				\item[$\bullet$] We can compare the coefficients of $\omgao^k$ in both sides of \eqref{refined total Chern class of N : equation 1} for $0 \le k \le m-2$;
				\item[$\bullet$] we can compare the coefficients of $\omgaa^k$ in both sides of \eqref{refined total Chern class of N : equation 2} for $0 \le 2k \le m-2$.
			\end{enumerate}
			In this way, we obtain the leftest and the rightest Diophantine equations among \eqref{Diophantine equations when 2n-2m <= m-2}.

		\subsection{Inequalities in {\texorpdfstring{$a,b$}{a,b}} and {\texorpdfstring{$c$}{c}}} \label{subsec3.4}

			% Comment
			% In this {subsection}, we derive some inequalities in $a,b,c$ and $\alpha_k$ (Proposition \ref{inequalities in a,b,c}, \ref{lower bound of c} and \ref{inequality : a^2 > 4b}) by using the numerical non-negativity of Chern classes and solving two refined equations \eqref{refined total Chern class of N : equation 1} and \eqref{refined total Chern class of N : equation 2} in one variable.

			\begin{definition}
				Let $Y$ be a non-singular variety.
				A cohomology class $\Gamma \in H^{2k} (Y, \Zbb)$ is \emph{numerically non-negative} if the intersection numbers $\Gamma \cdot Z$ are non-negative for all subvarieties $Z$ of $Y$ of dimension $k$.
			\end{definition}

			In our case when $Y = Gr(2,m)$, a cohomology class $\Gamma \in \cohomgr{m}{2k}$ is numerically non-negative if and only if when we write $\Gamma$ as a linear combination with respect to the basis \eqref{standard basis of cohomology group}, every coefficient in $\Gamma$ is non-negative.
			The following proposition tells us a sufficient condition for the numerical non-negativity of all Chern classes of vector bundles $\mathcal{E}$.

			\begin{proposition} [{\cite[Proposition 2.1 (\romannum{1})]{TA}}] \label{numerically non-negativity of Chern classes}
				Let $Z$ be a non-singular variety and let $\mathcal{E}$ be a vector bundle of arbitrary rank on $Z$ which is generated by global sections.
				Then each Chern class $c_i (\mathcal{E})$	of $\mathcal{E}$ is numerically non-negative for all $i = 1, 2, \cdots \dim (Z)$.
			\end{proposition}

			We find vector bundles on $Gr(2,m)$ satisfying the assumption of Proposition \ref{numerically non-negativity of Chern classes} and obtain inequalities in $a,b$ and $c$.

			\begin{lemma} \label{some holomorphic vector bundles whose Chern classes are numerically non-negative}
				Each of vector bundles $E = \varphi^* (\ckE (2,n)), \, \varphi^* (Q(2,n))$ and $N$ is generated by global sections and its Chern classes are all numerically non-negative.
			\end{lemma}

			\begin{proof}
				Note that a vector bundle $\mathcal{E}$ on $Z$ is generated by global sections if and only if there is a surjective bundle morphism from a trivial bundle on $Z$ (of any rank) to $\mathcal{E}$.
				By the short exact sequence after replacing $m$ in \eqref{short exact sequence 1} with $n$ and the pullback of its dual under $\varphi$, $E$ and $\varphi^* (Q(2,n))$ are generated by global sections.
				Moreover, by the short exact sequence \eqref{short exact sequence 3}, $T_{Gr(2,n)} \simeq Q(2,n) \otimes \ckE (2,n)$ is generated by global sections and from this, we can conclude that $N = \quotient{\varphi^* (T_{Gr(2,n)})}{T_{Gr(2,m)}}$ is generated by global sections.
				Hence, for each vector bundle $\mathcal{E}$ of $E, \ \varphi^* (Q(2,n))$ and $N$, $c_0 (\mathcal{E}) = 1 \ge 0$ and $c_1 (\mathcal{E}), \cdots , c_{2m-4} (\mathcal{E})$ are all numerically non-negative by Proposition \ref{numerically non-negativity of Chern classes}.
			\end{proof}

			% 페이지 넘기기
			\newpage

			\begin{proposition} \label{inequalities in a,b,c}
				Let $n \ge m \ge 4$.
				\begin{enumerate}
					\item[(a)] $a \ge 1, \ b \ge 0$ and $b+c \ge 0$ with $a^2 \ge b$.
					\item[(b)] If $m \ge 5$, then $a^2 \ge 2b$.
					\item[(c)] If $m \ge 6$, then $a^2 > 2b$.
					\item[(d)] For $n \le \frac{3m-2}{2}$, let $\alpha_k \, (0 \le k \le 2n-2m)$ be the integers which are given as in \eqref{definitions of alpha_k and beta_k}. Then $\alpha_k \ge 0$ for all $0 \le k \le 2n-2m$.
				\end{enumerate}
			\end{proposition}

			\begin{proof}
				(a) By Lemma \ref{some holomorphic vector bundles whose Chern classes are numerically non-negative}, each Chern class of $E$ and $N$ is numerically non-negative.
				By \eqref{Chern classes of E},
				\[
					c(E) = 1 + a \, \omgao + b \, \omgbo + (b+c) \, \omgaa,
				\]
				so $a,b$ and $b+c$ are non-negative.
				By Lemma \ref{total Chern class of N},
				\[
					c_1 (N) = (an-m) \, \omgao,
				\]
				so $a \ge 1$.

				By Lemma \ref{some holomorphic vector bundles whose Chern classes are numerically non-negative} again,
				\begin{align*}
					c_2 (\varphi^* (Q(2,n)))
					&= \varphi^* (\tlomgbo) \quad (\because \text{ Proposition \ref{total Chern classes of E(d,m) and Q(d,m)} (b)})\\
					&= \varphi^* (\tlomgao^2 - \tlomgaa) \quad (\because \text{ Corollary \ref{refined Pieri's formula}})\\
					&= (a \, \omgao)^2 - (b \, \omgao^2 + c \, \omgaa)\\
					&= (a^2-b) \, \omgao^2 - c \, \omgaa\\
					&= (a^2-b) \, \omgbo + (a^2-b-c) \, \omgaa
				\end{align*}
				is numerically non-negative.
				Since $m \ge 4$, $\omgbo$ is not zero, so we have $a^2 \ge b$.

				(b) By Lemma \ref{some holomorphic vector bundles whose Chern classes are numerically non-negative},
				\begin{align*}
					c_3 (\varphi^* (Q(2,n)))
					&= \varphi^* (\tilde{\omega}_{3,0}) \quad (\because \text{ Proposition \ref{total Chern classes of E(d,m) and Q(d,m)} (b)})\\
					&= \varphi^* (\tlomgao^3 - 2 \, \tlomgao \, \tlomgaa) \quad (\because \text{ Corollary \ref{refined Pieri's formula}})\\
					&= (a \, \omgao)^3 - 2a \, \omgao \, (b \, \omgao^2 + c \, \omgaa)\\
					&= a (a^2-2b) \, \omgao^3 -2ac \, \omgao \, \omgaa\\
					&= a (a^2-2b) \, \omega_{3,0} + 2a (a^2-2b-c) \, \omega_{2,1}
				\end{align*}
				is numerically non-negative.
				Since $m \ge 5$, $\omega_{3,0}$ is not zero, so we have $a(a^2-2b) \ge 0$.
				By (a), $a \ge 1$, thus $a^2 \ge 2b$.

				(c) By Lemma \ref{some holomorphic vector bundles whose Chern classes are numerically non-negative},
				\begin{align*}
					c_4 (\varphi^* (Q(2,n)))
					&= \varphi^* (\tilde{\omega}_{4,0}) \quad (\because \text{ Proposition \ref{total Chern classes of E(d,m) and Q(d,m)} (b)})\\
					&= \varphi^* (\tlomgao^4 - 3 \, \tlomgao^2 \, \tlomgaa + \tlomgaa^2) \quad (\because \text{ Corollary \ref{refined Pieri's formula}})\\
					&= (a \, \omgao)^4 - 3(a \, \omgao)^2 (b \, \omgao^2 + c \, \omgaa) + (b \, \omgao^2 + c \, \omgaa)^2\\
					&= (a^4-3a^2b+b^2) \, \omgao^4 + (-3a^2c+2bc) \, \omgao^2 \, \omgaa + c^2 \, \omgaa^2\\
					&= (a^4-3a^2b+b^2) \, \omega_{4,0} + \alpha \, \omega_{3,1} + \beta \, \omega_{2,2}\\
					& \hskip 1.5cm (\because \text{ Proposition \hyperlink{new basis and relation between two bases : target}{\ref{new basis and relation between two bases}} (b)})
				\end{align*}
				is numerically non-negative ($\alpha, \beta \in \Zbb [a,b,c]$). Since $m \ge 6$, $\omega_{4,0}$ is not zero, so we have $a^4-3a^2b+b^2 \ge 0$, that is, either $a^2 \ge \left( \frac{3+\sqrt{5}}{2} \right) b$ or $a^2 \le \left( \frac{3-\sqrt{5}}{2} \right) b$.
				But $a^2 \le \left( \frac{3-\sqrt{5}}{2} \right) b < b$ is impossible by (a). Hence, we have
				\[
					a^2 \ge \left( \frac{3+\sqrt{5}}{2} \right) b > 2b.
				\]

				(d) By Proposition \hyperlink{new basis and relation between two bases : target}{\ref{new basis and relation between two bases}} (b), the coefficient of $\omega_{k,0}$ in $c_k (N)$ with respect to the basis \eqref{standard basis of cohomology group} is equal to $\alpha_k$.
				By Lemma \ref{some holomorphic vector bundles whose Chern classes are numerically non-negative}, each $c_k (N)$ is numerically non-negative, thus $\alpha_k \ge 0$ for all $0 \le k \le 2n-2m$.
			\end{proof}

			Solving the refined equation \eqref{refined total Chern class of N : equation 2} in one variable $\omgaa$, each $\beta_{2k}$ with $1 \le k \le n-m$ is a polynomial in only one variable $c$.
			Comparing the coefficients of $\omgaa^k$ in both sides of \eqref{refined total Chern class of N : equation 2} for suitable powers $k$, we obtain a numerical condition on $c$.

			\begin{proposition} \label{lower bound of c}
				For $m \le n \le \frac{3m-6}{2}$, we have $c \ge 1$.
			\end{proposition}

			\begin{proof}
				We complete the proof by showing that the following two cases are impossible:
				\[
					\text{Case 1. } c \le -1;  \hskip 2.0cm \text{Case 2. } c=0.
				\]

				\begin{enumerate}
					\item[Case 1.] Suppose that $c \le -1$.
					By \eqref{refined total Chern class of N : equation 2}, we have
					\begin{eqnmg} \label{lower bound of c : equation 1}
						\left( \sum_{k=0}^{n-m} \beta_{2k} \, \omgaa^k \right) (1 + 4c \, \omgaa)\\
						= (1 + c \, \omgaa)^n (1 + \omgaa)^{-m} (1 + 4 \, \omgaa),
					\end{eqnmg}
					which is satisfied in $\quotientringaa$.
					Since $2(n-m+2) \le m-2$, we can compare the coefficients of $\omgaa^{n-m+2}$ in both sides of \eqref{lower bound of c : equation 1}.
					Thus we have
					\begin{eqnma} \label{lower bound of c : equation 2}
						0
						& = \sum_{k=0}^{n-m+2} \binom{n}{k} \binom{m+(n-m+2-k)-1}{n-m+2-k} \, c^k \, (-1)^{n-m+2-k}\\
						& \hskip 0.3cm + \, 4 \sum_{k=0}^{n-m+1} \binom{n}{k} \binom{m+(n-m+1-k)-1}{n-m+1-k} \, c^k \, (-1)^{n-m+1-k}\\
						& = \sum_{k=0}^{n-m+2} \binom{n}{k} \binom{n+1-k}{m-1} \, (-c)^k \, (-1)^{n-m+2}\\
						& \hskip 0.3cm + \, 4 \sum_{k=0}^{n-m+1} \binom{n}{k} \binom{n-k}{m-1} \, (-c)^k \, (-1)^{n-m+1}.
					\end{eqnma}
					After dividing both sides of \eqref{lower bound of c : equation 2} by $(-1)^{n-m+1}$,
					\begin{equation} \label{lower bound of c : equation 3}
						4 \sum_{k=0}^{n-m+1} \binom{n}{k} \binom{n-k}{m-1} (-c)^k = \sum_{k=0}^{n-m+2} \binom{n}{k} \binom{n+1-k}{m-1} (-c)^k.
					\end{equation}
					Since
					\begin{align*}
						\binom{n}{k} \binom{n-k}{m-1}
						&= \frac{n!}{k! \, (n-k)!} \cdot \frac{(n-k)!}{(m-1)! \, (n-m+1-k)!}\\
						&= \frac{n!}{(m-1)! \, (n-m+1)!} \cdot \frac{(n-m+1)!}{k! \, (n-m+1-k)!}\\
						&= \binom{n}{m-1} \binom{n-m+1}{k}
					\end{align*}
					for $0 \le k \le n-m+1$, the left hand side of \eqref{lower bound of c : equation 3} is equal to
					\begin{eqnma} \label{lower bound of c : equation 4}
						\text{(LHS)}
						&= 4 \sum_{k=0}^{n-m+1} \binom{n}{m-1} \binom{n-m+1}{k} (-c)^k\\
						&= 4 \, \binom{n}{m-1} (1-c)^{n-m+1}.
					\end{eqnma}
					Similarly, since
					\begin{align*}
						\binom{n}{k} \binom{n+1-k}{m-1}
						&= \frac{n!}{k! \, (n-k)!} \cdot \frac{(n+1-k)!}{(m-1)! \, (n-m+2-k)!}\\
						&= \frac{(n+1)!}{(m-1)! \, (n-m+2)!} \cdot \frac{(n-m+2)!}{k! \, (n-m+2-k)!} \cdot \frac{n+1-k}{n+1}\\
						&= \binom{n+1}{m-1} \binom{n-m+2}{k} \cdot \frac{n+1-k}{n+1}\\
						&\ge \binom{n+1}{m-1} \binom{n-m+2}{k} \cdot \frac{m-1}{n+1}\\
						&= \binom{n}{m-2} \binom{n-m+2}{k}
					\end{align*}
					for $0 \le k \le n-m+2$, the right hand side of \eqref{lower bound of c : equation 3} satisfies
					\begin{equation} \label{lower bound of c : equation 5}
						\text{(RHS)} \ \ge \ \binom{n}{m-2} (1-c)^{n-m+2}.
					\end{equation}
					Applying \eqref{lower bound of c : equation 4} and \eqref{lower bound of c : equation 5} to \eqref{lower bound of c : equation 3}, we have
					\[
						(1-c) \, \binom{n}{m-2} \ \le \ 4 \, \binom{n}{m-1}
					\]
					because $1-c \ge 0$.
					So $(1-c) \, (m-1) \ \le \ 4 \, (n-m+2)$, that is,
					\[
						4n \ge (5-c) \, m + c - 9.
					\]
					But since $4n \le 6m - 12$, we have
					\begin{equation} \label{lower bound of c : equation 6}
						(c+1) \, m \ge c+3.
					\end{equation}
					If $c=-1$, then \eqref{lower bound of c : equation 6} is impossible clearly. If $c \le -2$, then $c+1 \le -1$, so we have
					\[
						m \ \le \ \frac{c+3}{c+1} \ = \ 1 + \frac{2}{c+1} \ < \ 1
					\]
					by \eqref{lower bound of c : equation 6}, which implies a contradiction.

					\vskip 0.2cm

					\item[Case 2.] Suppose that $c=0$. Putting $c=0$ into \eqref{refined total Chern class of N : equation 2}, we have
					\begin{equation} \label{lower bound of c : equation 7}
						\sum_{k=0}^{n-m} \beta_{2k} \, \omgaa^k = (1 + \omgaa)^{-m} (1 + 4 \, \omgaa),
					\end{equation}
					which is satisfied in $\quotientringaa$.
					Comparing the coefficients of $\omgaa^{n-m+i}$ in both sides of \eqref{lower bound of c : equation 7} for $i=1$ and $2$, we have
					\begin{align*}
						0
						&= \binom{m+(n-m+i)-1}{n-m+i} (-1)^{n-m+i}\\
						& \hskip 1.0cm + 4 \, \binom{m+(n-m+i-1)-1}{n-m+i-1} (-1)^{n-m+i-1}\\
						&= (-1)^{n-m+i} \left\{ \binom{n+i-1}{m-1} - 4 \, \binom{n+i-2}{m-1} \right\}\\
						&= (-1)^{n-m+i} \binom{n+i-1}{m-1} \left( 1 - 4 \cdot \frac{n-m+i}{n+i-1} \right).
					\end{align*}
					So we obtain
					\[
						3n = 4m-3i-1
					\]
					for $i=1$ and $2$, which implies a contradiction.
				\end{enumerate}

				\vskip 0.2cm

				Hence, we conclude that $c \ge 1$.
			\end{proof}
			
			The refined equation \eqref{refined total Chern class of N : equation 1} in one variable $\omgao$ is harder than the refined equation \eqref{refined total Chern class of N : equation 2} in one variable $\omgaa$ to solve.
			To overcome this difficulty, we use a lower bound of $\alpha_{2n-2m}$ (Lemma \ref{lower bound of alpha_(2n-2m)}), which is obtained from Note \ref{result on the Chern classes of high degrees}, all the previous inequalities together and Lemma \ref{result on degree of omega_(i,j)}.
			This bound plays a key role in solving \eqref{refined total Chern class of N : equation 1}.

			\begin{lemma} [{\cite[Example 14.7.11]{FU} (or \cite[page 364]{HP})}] \label{result on degree of omega_(i,j)}
				In $Gr(2,m)$,
				\[
					\omega_{i,j} \, \omgao^{2m-4-i-j} = \frac{(2m-4-i-j)! \, (i-j+1)!}{(m-2-i)! \, (m-1-j)!}
				\]
				for $m-2 \ge i \ge j \ge 0$.
			\end{lemma}

			\begin{lemma} \label{lower bound of alpha_(2n-2m)}
				Let $m \le n \le \frac{3m-6}{2}$ and $\alpha_{2n-2m}$ be the integer given as in \eqref{definitions of alpha_k and beta_k}. Then
				\[
					\alpha_{2n-2m} \ge \frac{(2m-4)!}{(m-2)! \, (m-1)!} \cdot b^{n-2}.
				\]
			\end{lemma}

			\begin{proof}
				By Note \ref{result on the Chern classes of high degrees}, $c_{2n-2m} (N) = e (N_{\Rbb})$.
				So by Proposition \ref{Euler class of N_R}, $\alpha_{2n-2m}$ is equal to the coefficient of $\omgao^{2n-2m}$ in
				\[
					e(N_{\Rbb}) = \sum_{i=0}^{n-m} d_i \, \varphi^* (\tilde{\omega}_{2n-2m-i,i})
				\]
				where $d_i := X \cdot \tilde{\omega}_{n-2-i,2m-n-2+i}$ with respect to the basis \eqref{new basis of cohomology group of low degree} (that is, $\alpha_{2n-2m} = A_0$, which is given as in \eqref{naive expression of Euler class}).
				Here, $d_i \ge 0$ for all $0 \le i \le n-m$ because each $d_i$ is the intersection number of two subvarieties of $Gr(2,n)$.

				Let $\gamma_i$ be the coefficient of $\omgao^{2n-2m-2i}$ in $\varphi^* (\tilde{\omega}_{2n-2m-2i,0})$ with respect to \eqref{new basis of cohomology group of low degree}.
				Then the coefficient of $\omgao^{2n-2m}$ in
				\begin{align*}
					\varphi^* (\tilde{\omega}_{2n-2m-i,i})
					&= \varphi^* (\tilde{\omega}_{2n-2m-2i,0}) \, \varphi^* (\tlomgaa^i) \quad (\because \ \text{Corollary } \ref{refined Pieri's formula})\\
					&= \varphi^* (\tilde{\omega}_{2n-2m-2i,0}) \, (b \, \omgao^2 + c \, \omgaa)^i
				\end{align*}
				with respect to \eqref{new basis of cohomology group of low degree} is equal to $\gamma_i \, b^i$.
				By Proposition \hyperlink{new basis and relation between two bases : target}{\ref{new basis and relation between two bases}} (b), $\gamma_i$ is equal to the coefficient of $\omega_{2n-2m-2i,0}$ in $\varphi^* (\tilde{\omega}_{2n-2m-2i,0})$ with respect to the basis \eqref{standard basis of cohomology group}, and by Lemma \ref{some holomorphic vector bundles whose Chern classes are numerically non-negative}, it is non-negative for all $0 \le i \le n-m$.
				So we have
				\begin{eqnma} \label{lower bound of alpha_{2n-2m} : equation 1}
					\alpha_{2n-2m}
					&= \sum_{i=0}^{n-m} d_i \, \gamma_i \, b^i\\
					&\ge d_{n-m} \, \gamma_{n-m} \, b^{n-m} \quad (\because \ b \ge 0 \text{ by Proposition \ref{inequalities in a,b,c} (a)})\\
					&= d_{n-m} \, b^{n-m} \quad (\because \ \text{Since } \varphi^* (\tilde{\omega}_{0,0}) = \omega_{0,0}, \ \gamma_{n-m} = 1).
				\end{eqnma}
				Applying
				\begin{align*}
					d_{n-m}
					&= \varphi^* (\tlomgaa^{m-2})
					= (b \, \omgao^2 + c \, \omgaa)^{m-2}\\
					&= \sum_{i=0}^{m-2} \binom{m-2}{i} \, b^{m-2-i} \, c^i \, \omgao^{2m-4-2i} \, \omgaa^i\\
					&\ge b^{m-2} \, \omgao^{2m-4} \quad (\because \ b, \, c, \, \omgao^{2m-4-2i} \, \omgaa^i \ge 0\\
					& \hskip 2.8cm \text{ by Proposition \ref{inequalities in a,b,c} (a), \ref{lower bound of c} and Lemma \ref{result on degree of omega_(i,j)}})\\
					&= \frac{(2m-4)!}{(m-2)! \, (m-1)!} \cdot b^{m-2} \quad (\because \ \text{Lemma \ref{result on degree of omega_(i,j)}})
				\end{align*}
				to \eqref{lower bound of alpha_{2n-2m} : equation 1}, we obtain the desired inequality.
			\end{proof}

			Now, we ready to prove an inequality in $a$ and $b$, which is better than that of Proposition \ref{inequalities in a,b,c} (c).
			Before proving this, we use the following notation.

			\begin{notation} \label{definition of lmod}
				Let $R$ be a ring $\Zbb$ or $\Rbb$.
				Identifying $\quotient{R[x]}{(x^k)}$ with the free $R$-module which is generated by a basis $\{ 1, x, x^2, \cdots , x^{k-1} \}$, express an element in $\quotient{R[x]}{(x^k)}$ uniquely as a linear combination of $1, x, x^2, \cdots x^{k-1}$ with integral coefficients.
				We use the notation $f(x) \lmod g(x)$ in $\quotient{R[x]}{(x^k)}$ if the coefficients of $x^i$ in $g(x) - f(x)$ are non-negative for all $0 \le i < k$.
			\end{notation}

			\begin{proposition} \label{inequality : a^2 > 4b}
				For $m \le n \le \frac{3m-6}{2}$, let $a$ and $b$ be the integers which are given as in \eqref{Chern classes of E}.
				Then we have $a^2 > 4b$.
			\end{proposition}

			\begin{proof}
				Suppose that $a^2 \le 4b$.
				By Proposition \ref{inequalities in a,b,c} (a), $a \ge 1$, so we have $b \ge 1$.
				Then by \eqref{refined total Chern class of N : equation 1}, we have
				\begin{eqnma} \label{proposition 1 : equation 1}
					\sum_{k=0}^{2n-2m} \alpha_k \, \omgao^k
					&\lmod (1 + a \, \omgao + b \, \omgao^2)^n\\
					&\lmod (1 + \sqrt{b} \, \omgao)^{2n}
				\end{eqnma}
				in $\quotientringao$.
				Comparing the coefficients of $\omgao^{2n-2m}$ in both sides of \eqref{proposition 1 : equation 1},
				\begin{equation} \label{proposition 1 : equation 2}
					\alpha_{2n-2m} \ \le \ \binom{2n}{2n-2m} \, b^{n-m} \ = \ \binom{2n}{2m} \, b^{n-m}.
				\end{equation}
				By Lemma \ref{lower bound of alpha_(2n-2m)} and \eqref{proposition 1 : equation 2},
				\[
					\frac{(2m-4)!}{(m-2)! \, (m-1)!} \cdot b^{n-2} \ \le \binom{2n}{2m} \, b^{n-m},
				\]
				so we have
				\begin{eqnma} \label{proposition 1 : equation 3}
					b^{m-2}
					& \le \binom{2n}{2m} \cdot \frac{(m-2)! \, (m-1)!}{(2m-4)!}\\
					& \le \binom{3m-6}{2m} \cdot \frac{(m-2)! \, (m-1)!}{(2m-4)!} \quad(\because \ n \le 2m; \ 2n \le 3m-6)\\
					& = \frac{(3m-6)(3m-7) \cdots (2m+1)}{(2m-4)(2m-5) \cdots (m+3)} \cdot\\
					& \hskip 2.5cm \frac{(m-2)(m-3)(m-4)}{(m+2)(m+1)m} \cdot (m-5)\\
					& \le 2^{m-6} \cdot (m-5).
				\end{eqnma}
				Hence, we conclude that $b \le 2$.

				Since $m \le \frac{3m-6}{2}$, we have $m \ge 6$. By Proposition \ref{inequalities in a,b,c} (c), $2b < a^2 \le 4b$, so the only possible pair $(a,b)$ is $(2,1)$.
				Putting $(a,b) = (2,1)$ into \eqref{refined total Chern class of N : equation 1}, we have
				\begin{equation} \label{proposition 1 : equation 4}
					\sum_{k=0}^{2n-2m} \alpha_k \, \omgao^k = (1 + \omgao)^{2n-m+1} (1 - \omgao)
				\end{equation}
				which is satisfied in $\quotientringao$. Since $2n-2m+1 \le m-2$, we can compare the coefficients of $\omgao^{2n-2m+1}$ in both sides of \eqref{proposition 1 : equation 4}.
				Then we have
				\begin{align*}
					0
					&= \binom{2n-m+1}{2n-2m+1} - \binom{2n-m+1}{2n-2m}\\
					&= \binom{2n-m+1}{2n-2m+1} \left( 1 - \frac{2n-2m+1}{m+1} \right).
				\end{align*}
				So $\frac{2n-2m+1}{m+1} = 1$, that is, $3m = 2n \ (\le 3m-6)$ which implies a contradiction.
				Hence, we conclude that $a^2 > 4b$.
			\end{proof}

	\vskip 0.5cm

	\section{Characterization of linear embeddings} \label{sec4}

		Denote each assumption on \hyperlink{main theorem}{Main Theorem} as follows:
		\begin{enumerate}
			\item[$\bullet$] General case : $9 \le m$ and $n \le \frac{3m-6}{2}$;
			\item[$\bullet$] Special case : $4 \le m$ and $n=m+1$.
		\end{enumerate}

		\subsection{General case} \label{subsec4.1}

			% Comment
			% In this {subsection}, we
			We prove \hyperlink{main theorem}{Main Theorem} for the case when $9 \le m$ and $n \le \frac{3m-6}{2}$.

			\begin{theorem} \label{main theorem : general case}
				If $9 \le m$ and $n \le \frac{3m-6}{2}$, then any embedding $\varphi \colon Gr(2,m) \hookrightarrow Gr(2,n)$ is linear.
			\end{theorem}

			Before proving Theorem 4.1.1, we first prove the following inequalities in $a$ and $b$:
			% To prove Theorem \ref{main theorem : general case} efficiently, it is helpful to simplify the vector bundle $E$ on $Gr(2,m)$ by applying W. Barth and A. Van de Ven's results (Proposition \ref{result on holomorphic vector bundles on P^k of rank 2} and \ref{result on holomorphic vector bundles on Gr(d,m) of rank 2}). To apply these results to $E$, we need an upper bound of $a$.

			\begin{proposition} \label{upper bound of sqrt(a^2-4b) and a}
				Under the same assumption with Theorem \ref{main theorem : general case}, we have the following inequalities:
				\begin{enumerate}
					\item[(a)] $\sqrt{a^2-4b} < \frac{m-4}{3}$.
					\item[(b)] $a < \frac{m-4}{2}$.
				\end{enumerate}
			\end{proposition}

			\begin{proof}
				(a) First, by Proposition \ref{inequality : a^2 > 4b}, the expression $\sqrt{a^2-4b}$ is well defined. By \eqref{refined total Chern class of N : equation 1}, we have
				\begin{gather*}
					\left( \sum_{k=0}^{2n-2m} \alpha_k \, \omgao^k \right) (1 + \omgao)^{m-1} (1 + (4b-a^2) \, \omgao^2)\\
					= (1 + a \, \omgao + b \, \omgao^2)^n (1 - \omgao)
				\end{gather*}
				which is satisfied in $\quotientringao$. For convenience, let
				\begin{align*}
					\sum_{k=0}^{m-2} A_k \, \omgao^k &:= \left( \sum_{i=0}^{2n-2m} \alpha_i \, \omgao^i \right) (1 + \omgao)^{m-1};\\
					\sum_{k=0}^{m-2} B_k \, \omgao^k &:= (1 + a \, \omgao + b \, \omgao^2)^n.
				\end{align*}
				Then we have
				\[
					A_{k+2} + (4b-a^2) A_k = B_{k+2} - B_{k+1}
				\]
				for all $0 \le k \le m-4$.
				In particular, when $k = 2n-2m+2 \ (\le m-4)$,
				\begin{equation} \label{proposition 2 : equation 1}
					A_{2n-2m+4} + (4b-a^2) \, A_{2n-2m+2} = B_{2n-2m+4} - B_{2n-2m+3}.
				\end{equation}

				Suppose that $\sqrt{a^2-4b} \ge \frac{m-4}{3}$. We derive a contradiction by comparing the signs of both sides of \eqref{proposition 2 : equation 1}.

				\begin{enumerate}
					\item[$\bullet$] Since $A_k = \sum_{i=0}^{2n-2m} \alpha_i \binom{m-1}{k-i}$ for all $k \ge 2n-2m$, we have
					\begin{align*}
						\text{(LHS)}
						& = A_{2n-2m+4} + (4b-a^2) \, A_{2n-2m+2}\\
						& = \sum_{i=0}^{2n-2m} \alpha_i \left\{ \binom{m-1}{2n-2m+4-i} - (a^2-4b) \binom{m-1}{2n-2m+2-i} \right\}\\
						& = \sum_{i=0}^{2n-2m} \alpha_i \binom{m-1}{2n-2m+2-i} \cdot C_i
					\end{align*}
					where $C_i := \frac{(-2n+3m-3+i)(-2n+3m-4+i)}{(2n-2m+4-i)(2n-2m+3-i)} - (a^2-4b)$.
					Since
					\begin{align*}
						C_i
						& \le \frac{(m-3)(m-4)}{4 \cdot 3} - (a^2-4b)\\
						& \le \frac{(m-4)(-m+7)}{36} < 0 \qquad (\because \ m \ge 9)
					\end{align*}
					for all $0 \le i \le 2n-2m$ and since $\alpha_i \ge 0$ by Proposition \ref{inequalities in a,b,c} (d), the left hand side of \eqref{proposition 2 : equation 1} is negative.

					\vskip 0.2cm

					\item[$\bullet$] Since $B_k = \sum_{i=0}^{\lfloor k/2 \rfloor} \binom{n}{i} \binom{n-i}{k-2i} a^{k-2i} \, b^i$, we have
					\begin{align*}
						\text{(RHS)}
						= & B_{2n-2m+4} - B_{2n-2m+3}\\
						\ge & \sum_{i=0}^{n-m+1} \binom{n}{i} a^{2n-2m+3-2i} \, b^i \cdot \\
						& \hskip 1.0cm \left\{ \binom{n-i}{2n-2m+4-2i} a - \binom{n-i}{2n-2m+3-2i} \right\}\\
						= & \sum_{i=0}^{n-m+1} \binom{n}{i} a^{2n-2m+3-2i} \, b^i \binom{n-i}{2n-2m+4-2i} \cdot D_i
					\end{align*}
					where $D_i := a - \frac{2n-2m+4-2i}{-n+2m-3+i}$.
					Since
					\begin{align*}
						D_i
						& \ge \sqrt{a^2-4b} - \frac{2n-2m+4}{-n+2m-3}\\
						& \ge \frac{m-4}{3} - \frac{(3m-6)-2m+4}{-\frac{3m-6}{2}+2m-3} \quad \left( \because \ n \le \frac{3m-6}{2} \right)\\
						& = \frac{m-4}{3} - \frac{2m-4}{m}\\
						& = \frac{(m-5)^2-13}{3m} > 0 \qquad (\because \ m \ge 9)
					\end{align*}
					for all $0 \le i \le n-m+1$ and since $a,b \ge 0$ by Proposition \ref{inequalities in a,b,c} (a), the right hand side of \eqref{proposition 2 : equation 1} is positive.
				\end{enumerate}

				\vskip 0.2cm

				As a result, the equality \eqref{proposition 2 : equation 1} does not hold, thus this implies a contradiction.
				Hence, $\sqrt{a^2-4b} < \frac{m-4}{3}$.

				(b) Suppose that $a \ge \frac{m-4}{2}$.
				Then by (a),
				\begin{equation} \label{proposition 2 : equation 2}
					a^2-4b < \frac{(m-4)^2}{9} \le \left( \frac{2a}{3} \right)^2.
				\end{equation}

				Regard \eqref{refined total Chern class of N : equation 1} as an equation over $\Rbb$ (instead of over $\Zbb$). By \eqref{refined total Chern class of N : equation 1} and \eqref{proposition 2 : equation 2}, we have
				\begin{eqnma} \label{proposition 2 : equation 3}
					\sum_{k=0}^{2n-2m} \alpha_k \, \omgao^k
					& \lmod (1 + a \, \omgao + b \, \omgao^2)^n \left\{ \sum_{i=0}^{2m-4} \left(\ \sqrt{a^2-4b} \, \omgao \right)^i \right\}\\
					& \lmod \left(1 + \frac{a}{2} \, \omgao \right)^{2n} \left\{ \sum_{i=0}^{2m-4} \left( \frac{2a}{3} \, \omgao \right)^i \right\}
				\end{eqnma}
				in $\quotient{\Rbb [\omgao]}{(\omgao^{m-1})}$.
				Comparing the coefficients of $\omgao^{2n-2m}$ in both sides of \eqref{proposition 2 : equation 3},
				\begin{align*}
					\alpha_{2n-2m}
					& \le \sum_{i=0}^{2n-2m} \binom{2n}{i} \left( \frac{a}{2} \right)^i \left( \frac{2a}{3} \right)^{2n-2m-i}\\
					& = \left\{ \sum_{i=0}^{2n-2m} \binom{2n}{i} \left( \frac{1}{2} \right)^i \left( \frac{2}{3} \right)^{2n-i} \right\} \cdot \left( \frac{3}{2} \right)^{2m} a^{2n-2m}\\
					& \le \left\{ \frac{1}{2} \, \sum_{i=0}^{2n} \binom{2n}{i} \left( \frac{1}{2} \right)^i \left( \frac{2}{3} \right)^{2n-i} \right\} \cdot \left( \frac{3}{2} \right)^{2m} a^{2n-2m} \quad (\because \ 2n-2m \le n)\\
					& = \frac{81}{32} \left( \frac{1}{2} + \frac{2}{3} \right)^{2n} \cdot \left( \frac{3}{2} \right)^{2m-4} a^{2n-2m}\\
					& \le \frac{81}{32} \left( \frac{7}{6} \right)^{3m-6} \cdot \left( \frac{3}{2} \right)^{2m-4} a^{2n-2m} \quad (\because \ 2n \le 3m-6)\\
					& = \frac{81}{32} \left( \frac{7^3}{96} \right)^{m-2} a^{2n-2m}.
				\end{align*}
				By Lemma \ref{lower bound of alpha_(2n-2m)}, we have
				\begin{equation} \label{proposition 2 : equation 4}
					\frac{(2m-4)!}{(m-2)! \, (m-1)!} \cdot b^{n-2} \ \le \ \frac{81}{32} \left( \frac{7^3}{96} \right)^{m-2} a^{2n-2m}.
				\end{equation}
				By \eqref{proposition 2 : equation 2}, $b > \frac{5}{36} a^2$, so the left hand side of \eqref{proposition 2 : equation 4} satisfies
				\begin{align*}
					\frac{(2m-4)!}{(m-2)! \, (m-1)!} \cdot b^{n-2}
					&= \frac{(2m-4)(2m-5) \cdots m}{(m-2)(m-3) \cdots 2} \cdot b^{n-2}\\
					&\ge 2^{m-3} \cdot \left( \frac{5}{36} \right)^{n-2} a^{2n-4}\\
					&\ge 2^{m-3} \cdot \left( \frac{5}{36} \right)^{(3m-10)/2} a^{2n-4} \quad \left( \because \ n \le \frac{3m-6}{2} \right)\\
					&\ge \frac{1}{2} \left( \frac{36}{5} \right)^2 \cdot 2^{m-2} \cdot \left( \frac{5}{36} \right)^{(3m-6)/2} a^{2n-4}\\
					&> 24 \left( \frac{5 \sqrt{5}}{3 \cdot 6^2} \right)^{m-2} a^{2n-4},
				\end{align*}
				thus,
				\[
					24 \left( \frac{5 \sqrt{5} a^2}{3 \cdot 6^2} \right)^{m-2} \ < \ \frac{81}{32} \left( \frac{7^3}{96} \right)^{m-2}.
				\]
				So we have $\frac{5 \sqrt{5} a^2}{3 \cdot 6^2} \le \frac{7^3}{96}$, that is, $a^2 \le \frac{9 \cdot 7^3}{40 \sqrt{5}} \simeq 34.5137$, thus $a \le 5$.

				Since $4b < a^2 < \frac{36}{5}b$, the only possible pairs $(a,b)$ are
				\[
					(3,2); \quad (4,3); \quad (5,4); \quad (5,5); \quad (5,6).
				\]
				However they are all impossible by Lemma \ref{impossible pairs (a,b) of integers}, which is stated later. Hence, $a < \frac{m-4}{2}$ as desired.
			\end{proof}

			\begin{lemma} \label{corollary to Riemann-Roch Theorem}
				If $m \ge 7$, then $12$ divides $ab \, (a^2-b+3)$.
			\end{lemma}

			\begin{proof}
				By \cite[Lemma 4.10]{TA}, if $k \ge 5$ and $\mathcal{E}$ is a vector bundle on $\Pbb^k$ of rank $2$ with
				\[
					c(\mathcal{E}) = 1 + \alpha \, H + \beta \, H^2
				\]
				where $H$ is a hyperplane of $\Pbb^k$, then $\alpha \beta \, (\alpha^2 - \beta + 3)$ is divisible by $12$.
				In our case, $Gr(2,m)$ contains a Schubert variety $Y \simeq \Pbb^{m-2}$ of type $(m-2,0)$ and the total Chern class of the restriction of $E$ to $Y$ is
				\[
					c(E \big|_Y) = 1 + a \, H + b \, H^2.
				\]
				Since $m-2 \ge 5$, $ab \, (a^2-b+3)$ is divisible by $12$.
			\end{proof}

			\begin{lemma} \label{impossible pairs (a,b) of integers}
				The following pairs $(a,b)$ of integers are impossible:
				\begin{enumerate}
					\item[(a)] $(3,2), (4,3)$ and $(5,4)$ for $m \le n \le \frac{3m-4}{2}$;
					\item[(b)] $(5,5)$ for $7 \le m \le n$;
					\item[(c)] $(5,6)$ for $m \le n \le \frac{3m-2}{2}$.
				\end{enumerate}
			\end{lemma}

			\begin{proof}
				(a) In this case, $a=b+1$ with $b=2,3$ or $4$, so $1 + a \, \omgao + b \, \omgao^2 = (1 + \omgao) (1 + b \, \omgao)$ and $a^2-4b = (b-1)^2$.
				By \eqref{refined total Chern class of N : equation 1}, we have
				\begin{eqnmg} \label{impossible pairs (a,b) of integers : equation 1}
					\left( \sum_{k=0}^{2n-2m} \alpha_k \, \omgao^k \right) (1 - (b-1)^2 \, \omgao^2)\\
					= (1 + \omgao)^{n-m+1} (1 + b \, \omgao)^n (1 - \omgao)
				\end{eqnmg}
				which is satisfied in $\quotientringao$.
				Comparing the coefficients of $\omgao^{2n-2m+2}$ in both sides of \eqref{impossible pairs (a,b) of integers : equation 1},
				\begin{eqnma} \label{impossible pairs (a,b) of integers : equation 2}
					& -(b-1)^2 \, \alpha_{2n-2m}\\
					= & \hskip 0.3cm \sum_{k=0}^{n-m+1} \binom{n-m+1}{k} \left\{ \binom{n}{2n-2m+2-k} \, b^{2n-2m+2-k} \right.\\
					& \hskip 4.5cm \left. - \ \binom{n}{2n-2m+1-k} \, b^{2n-2m+1-k} \right\}\\
					= & \hskip 0.3cm \sum_{k=0}^{n-m+1} \binom{n-m+1}{k} \binom{n}{2n-2m+2-k} \, b^{2n-2m+1-k} \cdot A_k
				\end{eqnma}
				where $A_k := b - \frac{2n-2m+2-k}{-n+2m-1+k}$.
				By Proposition \ref{inequalities in a,b,c} (d), the left hand side of \eqref{impossible pairs (a,b) of integers : equation 2} is non-positive.
				On the other hand, since
				\begin{align*}
					A_k
					& \ge b - \frac{2n-2m+2}{-n+2m-1}\\
					& \ge b - \frac{(3m-4)-2m+2}{-\left( \frac{3m-4}{2} \right)+2m-1} \qquad \left( \because \ n \le \frac{3m-4}{2} \right)\\
					& = b - \frac{2m-4}{m+2}\\
					& > 0 \qquad (\because \ b=2, \, 3 \text{ or } 4)
				\end{align*}
				for $0 \le k \le n-m+1$, the right hand side of \eqref{impossible pairs (a,b) of integers : equation 2} is positive which implies a contradiction.

				(b) Since $ab \, (a^2-b+3) = 25 \cdot (25-5+3) = 25 \cdot 23$ is not divisible by $12$, this case is impossible by Lemma \ref{corollary to Riemann-Roch Theorem}.

				(c) Since $a^2-4b = 1$, we have by \eqref{refined total Chern class of N : equation 1},
				\begin{eqnma} \label{impossible pairs (a,b) of integers : equation 3}
					\sum_{k=0}^{2n-2m} \alpha_k \, \omgao^k
					&\lmod (1 + 5 \, \omgao + 6 \, \omgao^2)^n\\
					&\lmod \left( 1 + \frac{5}{2} \, \omgao \right)^{2n}
				\end{eqnma}
				in $\quotientringao$. Comparing the coefficients of $\omgao^{2n-2m}$ in both sides of \eqref{impossible pairs (a,b) of integers : equation 3},
				\begin{eqnma} \label{impossible pairs (a,b) of integers : equation 4}
					\alpha_{2n-2m}
					&\le \binom{2n}{2n-2m} \left( \frac{5}{2} \right)^{2n-2m}\\
					&= \binom{2n}{2m} \left( \frac{5}{2} \right)^{2n-4} \left( \frac{2}{5} \right)^{2m-4}
				\end{eqnma}
				Applying Lemma \ref{lower bound of alpha_(2n-2m)} to \eqref{impossible pairs (a,b) of integers : equation 4},
				\begin{eqnma} \label{impossible pairs (a,b) of integers : equation 5}
					\left( \frac{24}{25} \right)^{n-2}
					\le & \ \binom{2n}{2m} \cdot \frac{(m-2)! \, (m-1)!}{(2m-4)!} \left( \frac{2}{5} \right)^{2m-4}\\
					\le & \ 2^{m-6} \cdot (m-5) \left( \frac{2}{5} \right)^{2m-4}\\
					& (\because \ \text{The same argument with \eqref{proposition 1 : equation 3}})\\
					= & \ \frac{m-5}{16} \left( \frac{8}{25} \right)^{m-2}.
				\end{eqnma}
				But since $n \le \frac{3m-6}{2}$,
				\[
					\left( \frac{24}{25} \right)^{n-2} \ge \left( \frac{25}{24} \right)^2 \left\{ \left( \frac{24}{25} \right)^{3/2} \right\}^{m-2},
				\]
				thus by \eqref{impossible pairs (a,b) of integers : equation 5}, we have $\frac{48 \sqrt{6}}{125} = \left( \frac{24}{25} \right)^{3/2} \le \frac{8}{25}$ which implies a contradiction.
			\end{proof}

			\begin{proof} [Proof of Theorem \ref{main theorem : general case}]
				Let $Y \simeq \Pbb^{m-2}$ be a Schubert variety of $Gr(2,m)$ of type $(m-2,0)$.
				Then the total Chern class of the restriction of $E := \varphi^* (\ckE (2,n))$ to $Y$ is
				\[
					c(E \big|_Y) = 1 + a H + b H^2
				\]
				where $H$ is a hyperplane of $Y$.
				For any projective line $\ell$ in $Y$,
				\[
					(E \big|_Y) \big|_{\ell} = E \big|_{\ell} \simeq \mathcal{O}_{\ell} (a_1) \oplus \mathcal{O}_{\ell} (a_2)
				\]
				for some integers $a_1, a_2$ with $a_1 + a_2 = a$.
				Since $(E \big|_Y) \big|_{\ell}$ is generated by global sections by Lemma \ref{some holomorphic vector bundles whose Chern classes are numerically non-negative}, $a_1$ and $a_2$ are non-negative.
				So we have
				\[
					B(E \big|_Y) \le \frac{a-0}{2} < \frac{m-4}{4}
				\]
				by Proposition \ref{upper bound of sqrt(a^2-4b) and a} (b) (For the definition of $B(E \big|_Y)$, see \eqref{definition of B(E)}).
				Hence, $E \big|_Y$ is decomposable by Proposition \ref{result on holomorphic vector bundles on P^k of rank 2}.
				Since $Y \simeq \Pbb^{m-2}$ is arbitrary, $E$ is either decomposable or isomorphic to $E(2,m) \otimes L$ for some line bundle $L$ on $Gr(2,m)$ by Proposition \ref{result on holomorphic vector bundles on Gr(d,m) of rank 2}.

				\begin{enumerate}
					\item[Case 1.] $E \simeq E(2,m) \otimes L$ :
					Let $c(L) = 1 + r \, \omgao$ with $r \in \Zbb$.
					Then
					\[
						c(E(2,m) \otimes L) = 1 + (2r+1) \, \omgao + r(r+1) \, \omgao^2 + \omgaa,
					\]
					that is,
					\begin{equation} \label{theorem 1 : equation 1}
						a = 2r+1; \qquad b = r(r+1); \qquad c=1
					\end{equation}
					with $r \ge 0$ by Proposition \ref{inequalities in a,b,c} (a).
					After putting \eqref{theorem 1 : equation 1} into \eqref{refined total Chern class of N : equation 1}, we have
					\begin{equation} \label{theorem 1 : equation 2}
						\left( \sum_{k=0}^{2n-2m} \alpha_k \, \omgao^k \right) (1 + \omgao)^m
						= (1 + r \, \omgao)^n (1 + (r+1) \, \omgao)^n
					\end{equation}
					which is satisfied in $\quotientringao$.
					Comparing the coefficients of $\omgao^{2n-2m}$ in both sides of \eqref{theorem 1 : equation 2},
					\begin{equation} \label{theorem 1 : equation 3}
						\alpha_{2n-2m} \ \le \ \binom{2n}{2n-2m} (r+1)^{2n-2m} \ = \ \binom{2n}{2m} (r+1)^{2n-2m}.
					\end{equation}
					By Lemma \ref{lower bound of alpha_(2n-2m)}, we have
					\begin{eqnma} \label{theorem 1 : equation 4}
						(r(r+1))^{n-2}
						\le & \ \binom{2n}{2m} \cdot \frac{(m-2)! \, (m-1)!}{(2m-4)!} \cdot (r+1)^{2n-2m}\\
						\le & \ 2^{m-6} \cdot (m-5) (r+1)^{2n-2m}\\
						& \ (\because \ \text{The same argument with \eqref{proposition 1 : equation 3}}).
					\end{eqnma}
					After dividing both sides of \eqref{theorem 1 : equation 4} by $(r+1)^{2n-2m}$,
					\begin{equation}  \label{theorem 1 : equation 5}
						r^{n-2} (r+1)^{-n+2m-2} \le \frac{m-5}{16} \cdot 2^{m-2}.
					\end{equation}
					Since $-n+2m-2 \ge -\left( \frac{3m-6}{2} \right) +2m-2 = \frac{m+2}{2} > 0$, we have
					\begin{equation}  \label{theorem 1 : equation 6}
						r^{2m-4} \le r^{n-2} (r+1)^{-n+2m-2}.
					\end{equation}
					Combining \eqref{theorem 1 : equation 5} and \eqref{theorem 1 : equation 6}, $r^2 \le 2$, thus, $r = 0$ or $1$.

					If $r=1$, then $(a,b) = (3,2)$ by \eqref{theorem 1 : equation 1}, which is impossible by Lemma \ref{impossible pairs (a,b) of integers} (a).
					Hence, $r=0$.

					\vskip 0.2cm

					\item[Case 2.] $E \simeq L_1 \oplus L_2$ :
					Let $c(L_1) = 1 + r_1 \, \omgao$ and $c(L_2) = 1 + r_2 \, \omgao$ with $r_1, r_2 \in \Zbb$.
					Then
					\[
						c(L_1 \oplus L_2) = 1 + (r_1 + r_2) \, \omgao + r_1 r_2 \, \omgao^2,
					\]
					that is,
					\[
						a = r_1 + r_2; \qquad b = r_1 r_2; \qquad c=0.
					\]
					However $c=0$ cannot be happened by Proposition \ref{lower bound of c}.
				\end{enumerate}

				\vskip 0.2cm

				As a result, $E \simeq E(2,m) \otimes L$ with $c(L) = 1$, thus $(a,b,c) = (1,0,1)$.
				Hence, $\varphi \colon Gr(2,m) \hookrightarrow Gr(2,n)$ is linear by Proposition \ref{equivalent conditions for the linearity of an embedding}.
			\end{proof}

		\subsection{Special case} \label{subsec4.2}

			% Comment
			% In this {subsection}, we
			We prove \hyperlink{main theorem}{Main Theorem} for the case when $4 \le m$ and $n=m+1$.

			\begin{theorem} \label{main theorem : special case}
				Let $\varphi \colon Gr(2,m) \hookrightarrow Gr(2,m+1)$ be an embedding.
				\begin{enumerate}
					\item[(a)] If $m=4$, then any embedding $\varphi$ is either linear or twisted linear.
					\item[(b)] If $m \ge 5$, then any embedding $\varphi$ is linear.
				\end{enumerate}
			\end{theorem}

			Since $9 \le m = n-1$ satisfies the conditions $9 \le m$ and $n \le \frac{3m-6}{2}$, the result of Theorem \ref{main theorem : special case} for that case is already verified and we do not have to prove it again.
			However, we prove Theorem \ref{main theorem : special case} for all the cases without using any results in {Subsection} \ref{subsec4.1}.

			When $m$ is too small, we cannot apply some results in {Subsection} \ref{subsec3.3} and \ref{subsec3.4}.
			But since $\rank (N) = 2$ is small, we can compute the top Chern class of $N$ and the Euler class of $N_{\Rbb}$ by hand, and construct explicit Diophantine equations in $a,b$ and $c$.

			\begin{proposition} \label{Diophantine equations and inequalities for special case}
				Let $m \ge 4$.
				\begin{enumerate}
					\item[(a)] The following two equations hold:
					\begin{equation} \label{coefficient of omgbo in c_2(N) and e(N_R)}
						\binom{m+1}{2} \, (a-1)^2 + a^2 - 1 + b \, (m-3) = (a^2 - b) \, d_0 + b \, d_1
					\end{equation}
					and
					\begin{equation} \label{coefficient of omgaa in c_2(N) and e(N_R)}
						c \, (m-3) - m + 4 = c \, (d_1 - d_0)
					\end{equation}
					where $d_0 := \varphi^* (\tilde{\omega}_{m-1,m-3})$ and $d_1 := \varphi^* (\tilde{\omega}_{m-2,m-2})$.
					\item[(b)] If $m \ge 5$, then $c$ divides $m-4$.
					Moreover, $2b+2c-a^2 > 0$ and $c \ge 1$.
				\end{enumerate}
			\end{proposition}

			\begin{proof}
				(a) By Note \ref{result on the Chern classes of high degrees}, $c_2 (N) = e(N_{\Rbb})$, and we compute $e(N_{\Rbb})$ and $c_2 (N)$ in Proposition \ref{Euler class of N_R} and Lemma \ref{total Chern class of N}, respectively.
				Comparing the coefficient of $\omgao^2$ (resp. $\omgaa$) in $c_2 (N)$ with that in $e(N_{\Rbb})$, we obtain the desired equation \eqref{coefficient of omgbo in c_2(N) and e(N_R)} (resp. \eqref{coefficient of omgaa in c_2(N) and e(N_R)}).

				(b) If $m \ge 5$, then $c$ divides $m-4$ by \eqref{coefficient of omgaa in c_2(N) and e(N_R)} and in particular, $c \neq 0$.
				After dividing both sides of \eqref{coefficient of omgaa in c_2(N) and e(N_R)} by $c$,
				\begin{eqnma} \label{proposition 3 : equation 1}
					& m - 3 - \frac{m-4}{c}\\
					= & \hskip 0.3cm d_1 - d_0\\
					= & \hskip 0.3cm \{ (2b-a^2) \, \omgao^2 + 2c \, \omgaa \} (b \, \omgao^2 + c \, \omgaa)^{m-3}\\
					= & \hskip 0.3cm \{ (2b-a^2) \, \omgbo + (2b+2c-a^2) \, \omgaa \} \, (b \, \omgbo + (b+c) \, \omgaa)^{m-3}\\
					= & \hskip 0.3cm \sum_{i=0}^{m-3} \binom{m-3}{i} \, b^i \, (b+c)^{m-3-i} \cdot\\
					& \hskip 1.1cm \left\{ (2b-a^2) \, \omgbo^{i+1} \, \omgaa^{m-3-i} + (2b+2c-a^2) \, \omgbo^i \, \omgaa^{m-2-i} \right\}.
				\end{eqnma}
				Since $m - 3 - \frac{m-4}{c} \ge m - 3 -\frac{m-4}{1} = 1$, the right hand side of \eqref{proposition 3 : equation 1} is positive.
				Furthermore, since $b, \, b+c, \, a^2-2b \ge 0$ by Proposition \ref{inequalities in a,b,c} (a) and (b), and since $\omgbo^{i+1} \, \omgaa^{m-3-i}, \, \omgbo^i \, \omgaa^{m-2-i} \ge 0$, we have
				\[
					2b+2c-a^2 > 0.
				\]
				Hence, $c > \frac{1}{2} \, (a^2-2b) \ge 0$ by Proposition \ref{inequalities in a,b,c} (b) again.
			\end{proof}

			\vskip 0.5cm

			\begin{proof} [Proof  of Theorem \ref{main theorem : special case}]
				(a) Note that $\omgao^4 = 2$ and $\omgao^2 \, \omgaa = 1 = \omgaa^2$ in $\cohomgr{4}{8} \simeq \Zbb$.
				So $d_0 := \varphi^* (\tilde{\omega}_{m-1,m-3})$ and $d_1 := \varphi^* (\tilde{\omega}_{m-2,m-2})$ are given by
				\begin{eqnma} \label{proposition 4 : equation 1}
					d_0 &= b \, (a^2-b) + (b+c) \, (a^2-b-c);\\
					d_1 &= b^2 + (b+c)^2.
				\end{eqnma}
				Since $m=4$, we have by \eqref{coefficient of omgbo in c_2(N) and e(N_R)} and \eqref{coefficient of omgaa in c_2(N) and e(N_R)},
				\begin{equation} \label{proposition 4 : equation 2}
					10 \, (a-1)^2 + a^2 - 1 + b = (a^2-b) \, d_0 + b \, d_1
				\end{equation}
				and
				\begin{equation} \label{proposition 4 : equation 3}
					c = c \, (d_1 - d_0).
				\end{equation}

				\begin{enumerate}
					\item[Case 1.] $c \neq 0$ :
					By \eqref{proposition 4 : equation 1} and \eqref{proposition 4 : equation 3},
					\begin{eqnma} \label{proposition 4 : equation 4}
						1
						&= d_1 - d_0\\
						&= b \, (2b-a^2) + (b+c) \, (2b+2c-a^2).
					\end{eqnma}
					Applying \eqref{proposition 4 : equation 4} to \eqref{proposition 4 : equation 2}, we have
					\begin{eqnma} \label{proposition 4 : equation 5}
						11a^2 - 20a + 9 + b
						&= a^2 \, d_0 + b \, (d_1-d_0)\\
						&= a^2 \, d_0 + b.
					\end{eqnma}
					So $a$ divides $9$ and $a^2$ divides $-20a+9$, thus, $a = 1$. Furthermore, in this case,
					\begin{equation} \label{proposition 4 : equation 6}
						b \, (b-1) = (b+c) \, \left\{ 1 - (b+c) \right\}.
					\end{equation}
					Since the right hand side of \eqref{proposition 4 : equation 6} is not greater than $\frac{1}{4}$, $b=0$ or $1$.
					The only possible pairs $(b,c)$ satisfying \eqref{proposition 4 : equation 4} and \eqref{proposition 4 : equation 6} are $(0,1)$ and $(1,-1)$.
					Hence, $(a,b,c) = (1,0,1)$ or $(1,1,-1)$.

					\vskip 0.2cm

					\item[Case 2.] $c=0$ :
					Applying \eqref{proposition 4 : equation 1} to \eqref{proposition 4 : equation 2}, we have
					\begin{equation} \label{proposition 4 : equation 7}
						10 \, (a-1)^2 + a^2 - 1 + b = 2b \, \left\{ (a^2-b)^2 + b^2 \right\}.
					\end{equation}
					Suppose that $b > 0$.
					After dividing both sides of \eqref{proposition 4 : equation 7} by $b$,
					\[
						2b^2 \left\{ \left( \frac{a^2}{b} - 1 \right)^2 + 1 \right\} \ < \ 11 \cdot \frac{a^2}{b} + 1,
					\]
					that is,
					\begin{equation} \label{proposition 4 : equation 8}
						2b^2 \, t^2 - (4b^2+11) \, t + 4b^2-1 < 0
					\end{equation}
					where $t := \frac{a^2}{b}$.
					The discriminant of \eqref{proposition 4 : equation 8} is
					\begin{align*}
						D
						&= (4b^2+11)^2 - 4 \cdot 2b^2 \cdot (4b^2-1)\\
						&= -16b^4 + 96b^2 + 121\\
						&= -16 (b^2-3)^2 + 265.
					\end{align*}
					If $b \ge 3$, then $D < 0$, so \eqref{proposition 4 : equation 8} is impossible.
					For $b=1$ or $2$, we can show directly that \eqref{proposition 4 : equation 7} does not have any integral solution $a$.
					Hence, $b = 0 \, (=c)$, so $a=1$ by \eqref{proposition 4 : equation 7}.
				\end{enumerate}

				\vskip 0.2cm

				By Proposition \ref{equivalent conditions for the linearity of an embedding}, if $(a,b,c) = (1,0,1)$, then $\varphi$ is linear, and if $(a,b,c) = (1,1,-1)$, then $\varphi$ is twisted linear.
				To complete the proof, it suffices to show that the pair $(a,b,c) = (1,0,0)$ is impossible. In this case, by Lemma \ref{total Chern class of N}, we have
				\begin{equation} \label{proposition 4 : equation 9}
					c(N) = 1 + \omgao = c(E).
				\end{equation}
				After dividing both sides of \eqref{total Chern class of N : equation} by \eqref{proposition 4 : equation 9}, we have
				\begin{equation} \label{proposition 4 : equation 10}
					(1 - \omgao^2) (1 + \omgao + \omgaa)^4 = (1 + \omgao)^4 (1 - \omgao^2 + 4 \, \omgaa).
				\end{equation}
				Comparing the cohomology classes of degree $6$ in \eqref{proposition 4 : equation 10},
				\[
					\binom{4}{3} \, \omgao^3 + \binom{4}{2} \cdot 2 \, \omgao \, \omgaa - \binom{4}{1} \, \omgao^3
					= \binom{4}{3} \, \omgao^3 + \binom{4}{1} \, \omgao \, (-\omgao^2 + 4 \, \omgaa),
				\]
				and from this, we have
				\[
					(4 \cdot 2 + 12 - 4 \cdot 2) \, \omega_{2,1} = \left\{ 4 \cdot 2 + 4 \cdot (-2 + 4) \right\} \, \omega_{2,1}
				\]
				because $\omgao^3 = 2 \, \omega_{2,1}$. Hence, this implies a contradiction.

				(b) By Proposition \ref{inequalities in a,b,c} (b), $a^2-b \ge b$, so we have
				\begin{eqnma} \label{theorem 2 : equation 1}
					& \binom{m+1}{2} \, (a-1)^2 + a^2 - 1 + b \, (m-3)\\
					\ge& \hskip 0.3cm b \, (d_0 + d_1) \qquad (\because \ \text{Equation } \eqref{coefficient of omgbo in c_2(N) and e(N_R)})\\
					= & \hskip 0.3cm a^2 \, b \, \omgao^2 \, (b \, \omgao^2 + c \, \omgaa)^{m-3}\\
					\ge & \hskip 0.3cm a^2 \, b^{m-2} \, \omgao^{2m-4}\\
					&
					\begin{tabular}[t]{lll}
						($\because \ b, \, c, \, \omgao^{2m-4-2i} \, \omgaa^i \ge 0$ & by & Proposition \ref{inequalities in a,b,c} (a),\\
						& & \ref{Diophantine equations and inequalities for special case} (b) and Lemma \ref{result on degree of omega_(i,j)})
					\end{tabular}
					\\
					= & \hskip 0.3cm \frac{(2m-4)!}{(m-2)! \, (m-1)!} \cdot a^2 \, b^{m-2} \qquad (\because \ \text{Lemma \ref{result on degree of omega_(i,j)}}).
				\end{eqnma}
				The left hand side of \eqref{theorem 2 : equation 1} is less than
				\[
					a^2 \, \left\{ \binom{m+1}{2} + m - 2 \right\}
				\]
				because $a^2 \ge b$ by Proposition \ref{inequalities in a,b,c} (a).
				After dividing both sides of \eqref{theorem 2 : equation 1} by $a^2$,
				\[
					\binom{m+1}{2} + m - 2 \ge \frac{(2m-4)!}{(m-2)! \, (m-1)!} \cdot b^{m-2},
				\]
				thus we have
				\[
					b =
					\begin{cases}
						0 \text{ or } 1, & \text{if } m = 5 \text{ or } 6\\
						0, & \text{if } m \ge 7
					\end{cases}
					\centermark{2}{.}
				\]

				Assume that $b=1$ with $m=5$ or $6$.
				By Proposition \ref{inequalities in a,b,c} (b) and Proposition \ref{Diophantine equations and inequalities for special case} (b),
				\[
					2 \ = \ 2b \ \le \ a^2 \ < \ 2b+2c \ = \ 2+2c
				\]
				and $c \, (\ge 1)$ divides $m-4$.
				The only possible pair $(m,a,c)$ satisfying these properties is $(6,2,2)$.
				Putting $(m,a,b,c) = (6,2,1,2)$ into \eqref{theorem 2 : equation 1}, then
				\[
					27 \ = \ \binom{7}{2} + 6 \ \ge \ \frac{8!}{4! \cdot 5!} \cdot 4 \ = \ 56
				\]
				which implies a contradiction.

				Assume that $b=0$ with $m \ge 5$.
				Then by \eqref{coefficient of omgaa in c_2(N) and e(N_R)},
				\begin{align*}
					c \, (m-3) - m + 4
					&= c \, (-a^2 \, \omgao^2 + 2c \, \omgaa) \, (c \, \omgaa)^{m-3}\\
					&= (2c-a^2) \, c^{m-2}\\
					&\ge c^{m-2} \quad (\because \ \text{Proposition \ref{Diophantine equations and inequalities for special case} (b)}),
				\end{align*}
				thus $c=1$.
				By Proposition \ref{Diophantine equations and inequalities for special case} (b), $2-a^2 > 0$, so $a=1$.

				Hence, $(a,b,c) = (1,0,1)$ for $m \ge 5$, thus any embedding $\varphi \colon Gr(2,m) \hookrightarrow Gr(2,m+1)$ is linear by Proposition \ref{equivalent conditions for the linearity of an embedding}.
			\end{proof}

	\vskip 0.5cm

	% Bibliography
		\bibliographystyle{amsalpha}
		\bibliography{./main_embs.bib}

\providecommand{\bysame}{\leavevmode\hbox to3em{\hrulefill}\thinspace}
\providecommand{\MR}{\relax\ifhmode\unskip\space\fi MR }
\providecommand{\MRhref}[2]{%
  \href{http://www.ams.org/mathscinet-getitem?mr=#1}{#2}
}
\providecommand{\href}[2]{#2}
\begin{thebibliography}{BVdV74b}

\bibitem[Arr96]{AR}
E.~Arrondo, \emph{Subvarieties of {G}rassmannians}, 1996, Lecture note,\\
  {A}vailable at: \url{http://www.mat.ucm.es/~arrondo/trento.pdf}.

\bibitem[Bar77]{BA1}
W.~Barth, \emph{Some properties of stable rank-{$2$} vector bundles on {${\bf
  P}\sb{n}$}}, Math. Ann. \textbf{226} (1977), no.~2, 125--150.

\bibitem[Bry01]{BR}
R.~L. Bryant, \emph{Rigidity and quasi-rigidity of extremal cycles in compact
  hermitian symmetric spaces}, Preprint, ar{X}iv:math/0006186v2 [math.{DG}],
  2001.

\bibitem[BT82]{BT}
R.~Bott and L.~W. Tu, \emph{Differential forms in algebraic topology}, Graduate
  Texts in Mathematics, vol.~82, Springer-Verlag, New York-Berlin, 1982.

\bibitem[BVdV74a]{BV1}
W.~Barth and A.~Van~de Ven, \emph{A decomposability criterion for algebraic
  {$2$}-bundles on projective spaces}, Invent. Math. \textbf{25} (1974),
  91--106.

\bibitem[BVdV74b]{BV2}
\bysame, \emph{On the geometry in codimension {$2$} of {G}rassmann manifolds},
  pp.~1--35. Lecture Notes in Math., Vol. 412, Springer, Berlin, 1974.

\bibitem[Cho49]{CH}
W.-L. Chow, \emph{On the geometry of algebraic homogeneous spaces}, Ann. of
  Math. (2) \textbf{50} (1949), 32--67.

\bibitem[Fed65]{FE}
S.~Feder, \emph{Immersions and embeddings in complex projective spaces},
  Topology \textbf{4} (1965), 143--158.

\bibitem[Ful98]{FU}
W.~Fulton, \emph{Intersection theory}, second ed., Ergebnisse der Mathematik
  und ihrer Grenzgebiete. 3. Folge. A Series of Modern Surveys in Mathematics,
  vol.~2, Springer-Verlag, Berlin, 1998.

\bibitem[GH94]{GH}
P.~Griffiths and J.~Harris, \emph{Principles of algebraic geometry}, Wiley
  Classics Library, John Wiley \& Sons, Inc., New York, 1994, Reprint of the
  1978 original.

\bibitem[HP52]{HP}
W.~V.~D. Hodge and D.~Pedoe, \emph{Methods of algebraic geometry. {V}ol. {II}.
  {B}ook {III}: {G}eneral theory of algebraic varieties in projective space.
  {B}ook {IV}: {Q}uadrics and {G}rassmann varieties}, Cambridge, at the
  University Press, 1952.

\bibitem[Huh11]{HUS}
S.~Huh, \emph{On triple {V}eronese embeddings of {$\mathbb{P}^n$} in the
  {G}rassmannians}, Math. Nachr. \textbf{284} (2011), no.~11-12, 1453--1461.

\bibitem[Huy05]{HUD}
D.~Huybrechts, \emph{Complex geometry}, Universitext, Springer-Verlag, Berlin,
  2005, An introduction.

\bibitem[Mok08]{MO}
N.~Mok, \emph{Characterization of standard embeddings between complex
  {G}rassmannians by means of varieties of minimal rational tangents}, Sci.
  China Ser. A \textbf{51} (2008), no.~4, 660--684.

\bibitem[MS74]{MS}
J.~W. Milnor and J.~D. Stasheff, \emph{Characteristic classes}, Princeton
  University Press, Princeton, N. J.; University of Tokyo Press, Tokyo, 1974,
  Annals of Mathematics Studies, No. 76.

\bibitem[OSS11]{OSS}
C.~Okonek, M.~Schneider, and H.~Spindler, \emph{Vector bundles on complex
  projective spaces}, Modern Birkh\"auser Classics, Birkh\"auser/Springer Basel
  AG, Basel, 2011, Corrected reprint of the 1988 edition, With an appendix by
  S. I. Gelfand.

\bibitem[SU06]{SU}
J.~C. Sierra and L.~Ugaglia, \emph{On double {V}eronese embeddings in the
  {G}rassmannian {$\mathbb{G}(1,N)$}}, Math. Nachr. \textbf{279} (2006), no.~7,
  798--804.

\bibitem[Tan74]{TA}
H.~Tango, \emph{On {$(n-1)$}-dimensional projective spaces contained in the
  {G}rassmann variety {${\rm Gr}(n,\,1)$}}, J. Math. Kyoto Univ. \textbf{14}
  (1974), 415--460.

\bibitem[Wal97]{WA}
M.~F. Walters, \emph{Geometry and uniqueness of some extreme subvarieties in
  complex {G}rassmannians}, Ph.D. thesis, University of Michigan, 1997, p.~134.

\end{thebibliography}
	
\end{document}